\documentclass[12pt]{article}

\usepackage{amsmath, amsthm, amssymb}
\usepackage{enumerate}
\usepackage{pdflscape}
\usepackage{caption}

\usepackage{ifpdf}
\ifpdf
\usepackage[pdftex]{graphicx}
\else
\usepackage[dvips]{graphicx}
\fi
\usepackage{tikz}
 	 \usetikzlibrary{arrows,backgrounds}
\usepackage[all]{xy}

\usepackage{multicol}

\usepackage{tocvsec2}

\usepackage{bbm}

\input xy
\xyoption{all}

\usepackage[pdftex,plainpages=false,hypertexnames=false,pdfpagelabels]{hyperref}
\newcommand{\arxiv}[1]{\href{http://arxiv.org/abs/#1}{\tt arXiv:\nolinkurl{#1}}}
\newcommand{\arXiv}[1]{\href{http://arxiv.org/abs/#1}{\tt arXiv:\nolinkurl{#1}}}
\newcommand{\doi}[1]{\href{http://dx.doi.org/#1}{{\tt DOI:#1}}}

\newcommand{\googlebooks}[1]{(preview at \href{http://books.google.com/books?id=#1}{google books})}

\usepackage{xcolor}
\definecolor{dark-red}{rgb}{0.7,0.25,0.25}
\definecolor{dark-blue}{rgb}{0.15,0.15,0.55}
\definecolor{medium-blue}{rgb}{0,0,.8}
\definecolor{DarkGreen}{RGB}{0,150,0}
\definecolor{rho}{named}{red}
\hypersetup{
   colorlinks, linkcolor={purple},
   citecolor={medium-blue}, urlcolor={medium-blue}
}

\usepackage{longtable}
\usepackage{fullpage}

\theoremstyle{plain}
\newtheorem{thm}{Theorem}[section]
\newtheorem*{thm*}{Theorem}

\newtheorem*{cor*}{Corollary}

\newtheorem*{conj*}{Conjecture}
\newtheorem{lem}[thm]{Lemma}

\newtheorem{prop}[thm]{Proposition}

\newtheorem*{quest*}{Question}
\newtheorem*{claim*}{Claim}

\theoremstyle{definition}
\newtheorem{defn}[thm]{Definition}

\newtheorem{ex}[thm]{Example}
\newtheorem{sub-ex}[thm]{Sub-Example}
\newtheorem{rem}[thm]{Remark}
\newtheorem*{rem*}{Remark}

\newcommand{\comment}[1]{}

\newcommand{\be}{\begin{enumerate}[label=(\arabic*)]}
\newcommand{\ee}{\end{enumerate}}

\def\semicolon{;}
\def\applytolist#1{
    \expandafter\def\csname multi#1\endcsname##1{
        \def\multiack{##1}\ifx\multiack\semicolon
            \def\next{\relax}
        \else
            \csname #1\endcsname{##1}
            \def\next{\csname multi#1\endcsname}
        \fi
        \next}
    \csname multi#1\endcsname}

\def\calc#1{\expandafter\def\csname c#1\endcsname{{\mathcal #1}}}
\applytolist{calc}QWERTYUIOPLKJHGFDSAZXCVBNM;
\def\bbc#1{\expandafter\def\csname bb#1\endcsname{{\mathbb #1}}}
\applytolist{bbc}QWERTYUIOPLKJHGFDSAZXCVBNM;
\def\bfc#1{\expandafter\def\csname bf#1\endcsname{{\mathbf #1}}}
\applytolist{bfc}QWERTYUIOPLKJHGFDSAZXCVBNM;
\def\sfc#1{\expandafter\def\csname s#1\endcsname{{\sf #1}}}
\applytolist{sfc}QWERTYUIOPLKJHGFDSAZXCVBNM;
\def\fc#1{\expandafter\def\csname f#1\endcsname{{\mathfrak #1}}}
\applytolist{fc}QWERTYUIOPLKJHGFDSAZXCVBNM;

\newcommand{\noshow}[1]{}
\newcommand{\MR}[1]{}

\usetikzlibrary{shapes}
\usetikzlibrary{backgrounds}
\usetikzlibrary{decorations,decorations.pathreplacing,decorations.markings}
\usetikzlibrary{fit,calc,through}
\usetikzlibrary{external}
\tikzset{
	super thick/.style={line width=3pt}
}
\tikzset{
    quadruple/.style args={[#1] in [#2] in [#3] in [#4]}{
        #1,preaction={preaction={preaction={draw,#4},draw,#3}, draw,#2}
    }
} 
\tikzstyle{shaded}=[fill=red!10!blue!20!gray!30!white]
\tikzstyle{unshaded}=[fill=white]
\tikzstyle{empty box}=[circle, draw, thick, fill=white, opaque, inner sep=2mm]
\tikzstyle{annular}=[scale=.7, inner sep=1mm, baseline]
\tikzstyle{rectangular}=[scale=.75, inner sep=1mm, baseline=-.1cm]
\tikzstyle{mid>}=[decoration={markings, mark=at position 0.5 with {\arrow{>}}}, postaction={decorate}]
\tikzstyle{mid<}=[decoration={markings, mark=at position 0.5 with {\arrow{<}}}, postaction={decorate}]
\tikzstyle{over}=[double, draw=white, super thick, double=]

\usepackage{xifthen}  
\usepackage{breqn}
\usepackage{yfonts}
\usepackage{comment}
\usepackage{tikz-cd}

\newcommand{\mydotw}[1]{\begin{scope}[shift={#1}] \fill[shift only,white] (0,0) circle (1.5pt); \draw[shift only,thick]  (0,0) circle (1.5pt);   \end{scope}}

\title{\Large Triangle presentations and tilting modules for $\text{SL}_{2k+1}$}
\author{Corey Jones}
\begin{document}
\maketitle

\begin{abstract}
Triangle presentations are combinatorial structures on finite projective geometries which characterize groups acting simply transitively on the vertices of a locally finite building of type $\tilde{\text{A}}_{n-1}$ ($n\ge3$). From a type $\tilde{\text{A}}_{n-1}$ triangle presentation on a geometry of order $q$, we construct a fiber functor on the diagrammatic monoidal category $\text{Web}(\text{SL}^{-}_{n})$ over any field $\mathbbm{k}$ with characteristic $p\ge n-1$ such that $q \equiv 1$ mod $p$. When $\mathbbm{k}$ is algebraically closed and $n$ odd, this gives new fiber functors on the category of tilting modules for $\text{SL}_{n}$. 
\end{abstract}

\section{Introduction}

The purpose of this paper is to build new fiber functors on the category of tilting modules for $\text{SL}_{n}$ from combinatorial structures called \textit{triangle presentations}, which arise in the theory of affine buildings. Recall that affine buildings are a class of simplicial complexes which provide a geometric context for the study of semisimple algebraic groups over non-Archimedean local fields \cite{MR327923}. An important family of examples are the type $\tilde{\text{A}}_{n-1}$ Bruhat-Tits buildings associated to the groups $\text{PGL}(n,K)$, where $K$ is a field with discrete valuation $\nu$. 
The group $\text{PGL}(n,K)$ has a canonical action on its building which is transitive on vertices. It is natural to study discrete subgroups $\Gamma\le \text{PGL}(n,K)$ whose action on the vertices is \textit{simply} transitive. In this case, the building can be recovered as the flag complex of the Cayley graph of $\Gamma$ with respect to a set of generators which map any fixed vertex to its nearest neighbors. 

There is an abstract characterization of such $\Gamma$ in terms of \textit{triangle presentations} \cite{MR1320514}, \cite{MR1232965}. Triangle presentations are defined as purely combinatorial structures on finite projective geometries (Definition \ref{tridefn}). A group $\Gamma$ admits a type-rotating action on an abstract locally finite $\tilde{\text{A}}_{n-1}$ building which is simply transitive on the vertices if and only if it admits a triangle presentation of type $\tilde{\text{A}}_{n-1}$ \cite[Theorem 2.5]{MR1320514}. Triangle presentations may thus be viewed as combinatorial manifestations of the local structure of buildings in type $\tilde{\text{A}}$.

In this paper we establish a connection between the rich combinatorics of triangle presentations and type A representation theory. The category $\text{PolyWeb}(\text{GL}_{n})$ is a diagrammatic symmetric monoidal category defined over an arbitrary field $\mathbbm{k}$ which yields a presentation for the category of polynomial representations of $\text{GL}_{n}$ when $\mathbbm{k}$ is algebraically closed \cite[Definition 4.7, Remark 4.15]{MR4135417}. In this paper we consider a monoidal quotient of this category which we call $\text{Web}(\text{SL}^{-}_{n})$, obtained by adding an isomorphism from the determinant object in $\text{PolyWeb}(\text{GL}_{n})$ to the monoidal unit object and imposing certain compatibility conditions with the crossing generator (Section \ref{SLnQuotients}). When $n$ is odd and $\mathbbm{k}$ algebraically closed, the idempotent completion of $\text{Web}(\text{SL}^{-}_{n})$ is equivalent to the category of $\text{SL}_{n}$ tilting modules $\text{Tilt}(\text{SL}_{n})$ (Remark \ref{tilting}).

Recall a \textit{fiber functor} on a linear monoidal category is a monoidal functor \footnote{Our fiber functors are not necessarily braided.} to the category of finite dimensional vector spaces. The main result of this paper is Theorem \ref{TriangleSLnThm}, which we summarize as follows:

\medskip

\noindent (\textbf{Theorem} \ref{TriangleSLnThm}) Given a triangle presentation $\mathcal{T}$ of type $\tilde{\text{A}}_{n-1}$ on a finite projective geometry of order $q$, for any field $\mathbbm{k}$ of characteristic $p\ge n-1$ with $q\equiv 1 \ \text{mod}\ p$, we construct a fiber functor $\text{Web}(\text{SL}^{-}_{n})\rightarrow \text{Vec}$. When $n$ is odd and $\mathbbm{k}$ algebraically closed, these define fiber functors on $\text{Tilt}(\text{SL}_{n})$.

\medskip  

Our construction yields an infinite family of fiber functors on $\text{Tilt}(\text{SL}_{n})$ for a fixed $n$ and $\mathbbm{k}$. To our knowledge, these are the first examples of fiber functors on $\text{Tilt}(\text{SL}_{n})$, $n\ge 3$, with the interesting property that the underlying dimensions of the vector spaces assigned to objects are strictly larger than their usual dimension.\footnote{For $n=2$, there are many examples of this type, e.g. from a vector space $V$ whose dimension is $2$ mod p, together with a non-degenerate anti-symmetric bilinear form.} Indeed, for a fixed $n$ field $\mathbbm{k}$ with characteristic $p$, there are infinitely many primes $q\equiv 1\ \text{mod}\ p$, and thus infinitely many triangle presentations with which we can build our fiber functors. The linear dimensions of the vector spaces assigned by our functors to the defining representation, for example, are given by $[n]_{q}$, where $q$ is the order of the triangle presentation. This increases with $q$, and thus the linear dimension of these vector spaces can be arbitrarily large.

Recall a solution to the (parameter independent, quantum) Yang-Baxter equation consists of a vector space $V$ and an isomorphism $\check{R}:V\otimes V\rightarrow V\otimes V$ satisfying the equation$$
(\check{R}\otimes 1_{V})\circ (1_{V}\otimes \check{R})\circ (\check{R}\otimes 1_{V})= (1_{V}\otimes \check{R})\circ (\check{R}\otimes 1_{V})\circ (1_{V}\otimes \check{R})\ \text{in}\ \  \text{End}(V\otimes V\otimes V)
$$ Solutions to this equation have played an important role in statistical mechanics and the theory of quantum groups \cite{MR1062425}. In the latter context, they typically arise as $q$-deformations of the standard solution $P$, which swaps the order of the tensor factors. By considering the image of the crossing morphisms from $\text{Web}(\text{SL}^{-}_{n})$ under our (non-braided) fiber functor, we obtain new involutive ($\check{R}^{2}=\text{id}_{V\otimes V}$) solutions of the Yang-Baxter equation in positive characteristic. Our solutions can be interpreted as ``positive characteristic deformations'' of easy solutions in characteristic $0$ (Remark \ref{degenerate} and Section \ref{crossing}).

Our study of triangle presentations was inspired by and is closely related to \cite{MR3899967}, which uses $\tilde{\text{A}}_{2}$ triangle presentations to construct the first examples of genuinely quantum discrete quantum groups with property (T). Given a triangle presentation $\mathcal{T}$, if one considers the linear maps defining our functors over $\mathbbm{C}$, they no longer satisfy the $\text{SL}^{-}_{n}$ relations hence do not give a fiber functor of $\text{Web}(\text{SL}^{-}_{n})$. In this setting, we can upgrade our vector spaces to Hilbert spaces by asserting the basis elements $\Pi$ be orthonormal. Then these maps generate a rigid C*-tensor subcategory of finite dimensional Hilbert spaces. By Tannaka-Krein-Woronowicz duality this yields a corresponding compact quantum group. In the $n=3$ case, it is easy to see from the definitions that this is precisely the compact quantum group $\mathbbm{G}_{\mathcal{T}}$ introduced in \cite[Definition 5.1]{MR3899967} associated to $\mathcal{T}$. Another way to understand this connection is to say the planar algebra introduced in \cite[Section 7]{MR3899967} (which makes sense over $\mathbbm{Z}$) describing $\text{Rep}(\mathbbm{G}_{\mathcal{T}})$ satisfies the $\text{SL}_{3}$ web relations in when reduced modulo $p$ dividing $q-1$.

We briefly describe the outline of the paper. In Section \ref{prelims}, we cover background material including finite projective geometries, buildings, triangle presentations and examples. No knowledge of buildings is required to read this paper, but we include some basic definitions and discussion for the interested reader in Section \ref{buildings}. In Section \ref{categories} we describe the relevant web categories and some basic results concerning their diagrammatics. In Section \ref{functor}, we construct the functors and then discuss basic properties of the resulting Yang-Baxter solution $\check{R}$. The main technical part of the paper is the proof of Lemma \ref{squarelemma1}. This involves a combinatorial case analysis, and makes full use of all the axioms of triangle presentations. 

\subsection*{Acknowledgements}

We would like to thank Pavel Etingof, Aaron Lauda, Victor Ostrik, David Penneys, and Sean Sanford for useful comments and discussions. We thank Emily McGovern for discovering numerous typos. This research was supported by NSF Grant DMS-1901082.

\section{Preliminaries: Projective geometry, buildings, and triangle presentations}\label{prelims}

Here we recall some basic definitions and results about finite projective geometries, buildings, and triangle presentations. As triangle presentations can be defined purely as a combinatorial structure on a finite projective geometry, background in the theory of buildings in not necessary to read this paper. However, we include here the basic definitions, since the connection between triangle presentations and groups acting on buildings is their raison d'etre.

\subsection{Finite projective geometry}

Finite projective geometries are examples of \textit{incidence geometries}. They consist of a finite set of points $P$, a finite set of lines $L$, and an \textit{incidence relation} between points and lines. For a comprehensive general reference on incidence geometries, we refer the reader to the notes \cite{Moorhouse07}.

A finite projective geometry of projective dimension $n$ is an incidence geometry satisfying the following axioms (called the Veblen-Young axioms, for example \cite[pp.127]{Moorhouse07}):

\begin{enumerate}

\item
For every two distinct points $p,q\in P$ there is a unique line $p\vee q$ incident with both $p$ and $q$
\item
There exists three non-collinear points
\item
Every line is incident with at least $3$ points.
\item
The maximum length of a chain of (non-empty) subspaces is $n+1$
\item
Every line incident with two sides of a triangle but not incident with its vertices must be incident with the third side of the triangle
\end{enumerate}

\noindent Here \textit{subspace} is a collection of points such that for any two points the line containing them is also in the collection (we consider singleton sets of points subspaces). In the definition of chains of subspaces, we assume inclusions are proper and do not count $\varnothing$ as a subspace.

Given a subspace $V$, we can define the \textit{algebraic dimension} to be the largest $k$ so that there exists a chain of proper inclusions of (proper) subspaces $V_{0}\subseteq V_{1}\subseteq \cdots \subseteq V_{k-1}\subseteq V$, while the \textit{projective dimension} is given by $k-1$. Directly from the definitions we see the projective dimension of the entire set $P$ is $n$, and the algebraic dimension is $n+1$. It's not hard to show that all lines have the same number of points, and the \textit{order} $q$ of a projective geometry is defined to be one less than the number of points on a line.

The terminology ``algebraic dimension" stems from the canonical examples of finite projective geometries. Let $V$ be a vector space of dimension $n+1$ over a finite field $\mathbbm{F}_{q}$. The points in this projective geometry are the 1-dimensional subspaces, and the lines are 2-dimensional subspaces. Projective geometries arising from a finite vector space in this way are called \textit{Desarguesian}, or \textit{classical}. In this case, the subspaces of \textit{algebraic dimension} $k$ are precisely the linear subspaces of $V$ with (linear) dimension $k$. The order in this case is the size of the field $q$. We note that projective geometries with projective dimension $n\ge 3$ are always classical, while there are many exotic examples of projective planes (n=2).

In what follows, we will find it more convenient to index subspaces by their algebraic dimension rather than their geometric dimension. However, we will continue to specify algebraic dimension to avoid confusion. Given a projective geometry, we define $\Pi_{k}$ to be the set of subspaces of algebraic dimension $k$, so that the set of points is written $\Pi_{1}$, the set of lines is $\Pi_{2}$, etc. 

Let $q\ne 1$ be a fixed positive integer. Then for $k$ a positive integer, define $$[k]_{q}:=\frac{q^{n}-1}{q-1}=1+q+q^{2}+\cdots + q^{k-1}.$$ Then set $[0]_{q}=1$ and recursively define $[k+1]_{q}!:=[k+1]_{q}[k]_{q}!.$ For $l\le k$ the $q$-binomial coefficient is given by $$\left[ {\begin{array}{c}
   k  \\
   l  \\
  \end{array} } \right]_{q}
:=\frac{[k]_{q}!}{[l]_{q}![k-l]_{q}!}$$

\noindent We adopt the conventions $$[-k]_{q}=-[k]_{q}\ \ \text{and}\ \ [-k]_{q}!=[-k]_{q}[-(k-1)]_{q}\cdots [-1]_{q}=(-1)^{k}[k]_{q},$$ 

\noindent and then use the same definition for the binomial coefficients with (possibly) negative integers. The following easy lemma is well known (for example, \cite[pp.121]{Moorhouse07})

\begin{lem}\label{subspacenumber} For a finite projective geometry of order $q$ and algebraic dimension $n$ the number of subspaces of algebraic dimension $k$ is $\left[ {\begin{array}{c}
   n  \\
   k  \\
  \end{array} } \right]_{q}$.

\end{lem}

For $n\ge 4$, since all such geometries are classical this is an exercise in linear algebra. For $n=3$, this is a basic fact about projective planes (for example, \cite[Theorem 6.3]{Moorhouse07}).

\begin{lem}\label{intermediatesubspacenumber} For a finite projective geometry of order $q$ and algebraic dimension $n$ the number of subspaces of algebraic dimension $k$ containing a fixed subspace $V$ of dimension $m$ is $\left[ {\begin{array}{c}
   n-m  \\
   k-m  \\
  \end{array} } \right]_{q}$.

\end{lem}

\begin{proof}
For $n=3$, this again is a basic fact from the theory of projective planes. For $n\ge 4$, projective geometries are classical, and thus subspaces are linear subspaces of an $n$ dimensional vector space $W$ over $\mathbbm{F}_{q}$, whose algebraic dimension is their linear dimension. The first isomorphism theorem for vector space establishes a bijection between the intermediate subspaces $U$, $V\subseteq U\subseteq W$, of dimension $k$, and arbitrary subspaces of the $n-m$ dimensional quotient space $W/V$ of dimension $k-m$. Applying the previous Lemma gives the result.
\end{proof}

Finally, an easy fact we will make frequent use of is if $q\equiv 1\ \text{mod}\ p$  $$\left[ {\begin{array}{c}
   n  \\
   k  \\
  \end{array} } \right]_{q}\equiv \left( {\begin{array}{c}
   n  \\
   k  \\
  \end{array} } \right)\ \text{mod}\ p$$

\subsection{Buildings}\label{buildings}

Ultimately for our construction we don't need the full apparatus of buildings, we only need the combinatorial properties of triangle presentations. However, the origins of triangle presentations lie in the study of symmetries of buildings and we believe this geometric context is an interesting aspect of our construction. Therefore we find it prudent to take a brief detour to discuss buildings for the unfamiliar reader. We refer to the reader to \cite{MR2560094}, \cite{MR2439729} for a comprehensive reference for the topics discussed below.

Buildings, originally introduced by Tits, are a family of highly symmetric simplicial complexes whose geometric realizations have nice properties. For semisimple algebraic groups $G$ over non-Archimedean local fields, they play a role analogous to homogeneous spaces. In particular, such a group admits a vertex transitive action on its corresponding Bruhat-Tits building, allowing for the application of geometric methods to study $G$.

More concretely, buildings are simplicial complexes which are modeled on \textit{Coxeter complexes}. Indeed, Coxeter complexes are precisely the degenerate (or thin) buildings, forming the fundamental pieces of (thick) buildings which will be of primary interest. \textit{Coxeter complexes} are simplicial complexes associated to a Coxeter system $(W,S)$ on which the underlying group $W$ acts in a nice way. As Coxeter complexes are widely studied throughout mathematics, we neglect to include  definitions but refer the reader to \cite[Chapters 2, 3]{MR2439729}, which is especially useful for the theory of buildings.

\begin{defn}\label{buildingdef}

If $(W,S)$ is a Coxeter system, a \textit{building} of type $(W,S)$ is a simplicial complex $\Delta$ which is a union of subcomplexes called \textit{apartments} satisfying the following conditions:

\begin{enumerate}
    \item 
    Each apartment is a Coxeter complex of type $(W,S)$.
    \item
    Any two simplices are contained in a common apartment.
    \item
    If two simplices $S,T$ are both contained in apartments $A,A^{\prime}$, there is an isomorphism of chamber complexes $\phi:A\rightarrow A^{\prime}$ that fixes $S,T$ point-wise.

\end{enumerate}

\end{defn}

A chamber is a maximal simplex. Often in the literature, there is an additional assumption that the building is \textit{thick}, which means every codimenion $1$ simplex of a chamber is incident with at least 3 distinct chambers. Otherwise the building is called \textit{thin}. Coxeter complexes themselves are precisely the thin buildings.
 
We are interested in buildings of type $\tilde{\text{A}}_{n-1}$, $n\ge 3$. The associated Coxeter complexes are tessellations of Euclidean space by simplices. A large class of examples are the buildings $\Delta^{n}_{K}$ associated to a field $K$ with discrete valuation $\nu: K\rightarrow \mathbbm{Z}$. These are the Bruhat-Tits buildings associated to the group $\text{PGL}(n,K)$. We will provide here an elementary construction following \cite[Section 9.2]{MR2560094}.

Let $\mathcal{O}=\{x\in K\ :\ \nu(x)\ge 0\}$ denote the discrete valuation ring. Pick $\pi\in \mathcal{O}$ with $\nu(\pi)=1$. Then $\pi$ generates the unique maximal ideal in $\mathcal{O}$. In order to obtain a locally finite building, we assume the residue field $\mathcal{O}/\mathcal{O}\pi $ is finite. A \textit{lattice} in $K^{n}$ is a free rank $n$ $\mathcal{O}$ submodule $L\subseteq K^{n}$. We say two lattices $L,L^{\prime}$ are \textit{equivalent} if $\lambda L=L^{\prime}$ for some $\lambda\in K$. The vertices of $\Delta^{n}_{K}$ are equivalence classes of lattices. Two classes will have an edge between them if there exists{} representatives $L,L^{\prime}$ respectively such that $\pi L \subsetneq L^{\prime}\subsetneq L$. This gives us a graph, and the simplicial complex $\Delta^{n}_{K}$ is defined as the flag complex of this graph. In other words, the simplices of $\Delta^{n}_{K}$ are subsets of vertices such that every pair of elements has an edge between them. These buildings naturally admit an action of $\text{PGL}(n,K)$ which is transitive on the vertices. For $n\ge 4$, every $\tilde{\text{A}}_{n-1}$ building is isomorphic to one of this form.

Every $\tilde{\text{A}}_{n-1}$ building admits an essentially unique type function $\tau$ from the set of vertices to $\mathbbm{Z}/ n\mathbbm{Z}$. In the Bruhat-Tits example described above, we set $L_{0}:=\mathcal{O}^{n}\subseteq K^{n}$. Then any other lattice can be written as $gL_{0}$ for some $g\in \text{GL}(n,K)$. We can define the type function $\tau([gL_{0}]):\equiv \nu(\text{det}(g)) \ \text{mod}\ n\  \in \mathbbm{Z}/n\mathbbm{Z}$, which is easily seen to be well defined. An automorphism of a building $\alpha$ is said to be \textit{type rotating} if there exists a $c\in \mathbbm{Z}/n\mathbbm{Z}$ such that $\tau(\alpha(v))\equiv c+\tau(v)\ \text{mod}\ n$ for all vertices $v$ of the building. $\text{PGL}(n,K)$ acts on $\Delta_{K}$ by type rotating automorphisms.
 
 One important feature of buildings of type $\tilde{\text{A}}_{n-1}$ is that the collection of vertices neighboring a given vertex v (the link of the vertex) naturally has the structure of a type $\text{A}_{n-1}$ building, which in turn always carries the structure of a projective geometry of algebraic dimension $n$. If the building is locally finite, then this is a finite projective geometry.

 As explained in the introduction, if a connected, locally finite building $\Delta$ admits an action by a group $\Gamma$ which is simply transitive on the set of vertices, then picking a vertex $v$ we can identify the vertices adjacent to $v$ with a set of generators for $\Gamma$. We can then use the combinatorial axioms of the building to work out a presentation of the group in terms of these generators. Collecting properties of such a presentation for arbitrary buildings of type $\tilde{\text{A}}_{n-1}$ leads naturally to the definition of a triangle presentation over a finite projective geometry. First we give some notation and conventions.
 
 Let $\Pi$ be a finite projective geometry of algebraic dimension $n$. We use the notation $$\displaystyle \Pi:=\bigsqcup_{1\le k\le n} \Pi_{k}.$$ We sometimes abuse notation, and denote the projective geometry simply by $\Pi$. For $u\in \Pi$, we also introduce the notation $\text{dim}(u)$ for the \textit{algebraic dimension} of $u$. For subspaces $u,v\in \Pi$, we say $u$ \textit{is incident with} $v$, written $u\sim v$ if either $u\subseteq v$ or $v\subseteq u$. The following definition is from \cite{MR1320514}.

\begin{defn}\label{tridefn}
Let $\Pi$ be a finite projective geometry of algebraic dimension $n$ and $\sigma:\Pi\rightarrow \Pi$ an involution such that $\sigma(\Pi_{k})=\Pi_{n-k}$. A \textit{triangle presentation of type $\tilde{\text{A}}_{n-1}$} compatible with $\sigma$ consists of a collection $\mathcal{T}\subseteq \Pi\times \Pi \times \Pi$ satisfying the following conditions:

\begin{enumerate}
    \item\label{incidence}
    For $u,v\in \Pi$, $(u,v,w)\in \mathcal{T}$ for some $w$ if and only if $\sigma(u)$ and $v$ are distinct and incident.
    \item\label{cyclicinvariance}
    For $u,v,w\in \Pi$, $(u,v,w)\in \mathcal{T}$ if and only if $(v,w,u)\in \mathcal{T}$
    \item\label{mod}
    If $(u,v,w)\in \mathcal{T}$, then $\text{dim}(u)+\text{dim}(v)+\text{dim}(w)= 0\ \text{mod}\ n $.
    \item\label{uniqueness}
    If $(u,v,w_{1})\in \mathcal{T}$ and $(u,v,w_{2})\in \mathcal{T}$, then $w_{1}=w_{2}$.
    \item\label{lambdareflection}
    If $(u,v,w)\in \mathcal{T}$, then $(\sigma(w), \sigma(v), \sigma(u))\in \mathcal{T}$.
    \item\label{theweirdone}
    If $(u_{1},v_{1},w)\in \mathcal{T}$, $(u_{2}, v_{2}, \sigma(w))\in \mathcal{T}$ and $\text{dim}(u_{i})+\text{dim}(v_{i})<n$, then there exists a unique $z$ such that $(v_{2},u_{1}, z)\in \mathcal{T}$ and $(v_{1}, u_{2}, \sigma(z))\in \mathcal{T}$.
\end{enumerate}

 \end{defn}

\begin{prop}\label{betterweirdone} Condition \ref{theweirdone} in the above definition can be replaced by either of the following conditions:
\end{prop}

\begin{enumerate}

\item[$6^{\prime}$.]\label{betterweirdone1}
If $(u,v,\sigma(r)),(r,w,\sigma(s))\in \mathcal{T}$ and $\text{dim}(u)+\text{dim}(v)+\text{dim}(w)<n$, then there exists a unique $t\in \Pi$ such that $(u,t,\sigma(s)),(v,w,\sigma(t))\in \mathcal{T}$.

\item[$6^{\prime\prime}$.]\label{betterweirdone2}
If $(u,t,\sigma(s)),(v,w,\sigma(t))\in \mathcal{T}$ and $\text{dim}(u)+\text{dim}(v)+\text{dim}(w)<n$, then there exists a unique $r\in \Pi$ such that $(u,v,\sigma(r)),(r,w,\sigma(s))\in \mathcal{T}$.

 \end{enumerate}

 \begin{proof}
First we show \ref{theweirdone} is equivalent to $6^{\prime}$, assuming all the other conditions. If we perform the substitution $u\mapsto u_{1},v\mapsto v_{1},r\mapsto \sigma(w),s\mapsto \sigma(v_{2}), w\mapsto u_{2}$ and use $\ref{cyclicinvariance}$, then the statements are precisely the same, ignoring the conditions on the dimensions. Therefore it remains to show the assumptions on dimensions in $\ref{theweirdone}$ and $6^{\prime}$ respectively are equivalent. Under the substitution described above, supposing $\text{dim}(u_{i})+\text{dim}(v_{i})<n$ as in \ref{theweirdone}, then $\text{dim}(u)+\text{dim}(v)<n$ and $\text{dim}(w)+\text{dim}(\sigma(s))<n$. But note $\text{dim}(r)=\text{dim}(u)+\text{dim}(v)\ \text{and}\ \text{dim}(r)+\text{dim}(w)+\text{dim}(\sigma(s))\equiv 0\ \text{mod}\ n$ and thus $\text{dim}(w)+\text{dim}(\sigma(s))<n$ implies $\text{dim}(u)+\text{dim}(v)+\text{dim}(w)=\text{dim}(r)+\text{dim}(w)<n$ as desired. 

Conversely if we suppose $\text{dim}(u)+\text{dim}(v)+\text{dim}(w)<n$, then under the substitution described above we see $\text{dim}(\sigma(w))+\text{dim}(u_{2})=\text{dim}(u_1)+\text{dim}(v_1)+\text{dim}(u_{2})<n$ which immediately implies $\text{dim}(u_1)+\text{dim}(v_1)<n$. Also, since $\text{dim}(\sigma(w))+\text{dim}(u_{2})+\text{dim}(v_{2})\equiv 0\ \text{mod}\ n$ and the sum of the first two terms is less than $n$, then the whole sum must be $n$ hence the sum of the second two terms is less than $n$ as desired.

We will show $6^{\prime}$ implies $6^{\prime\prime}$, assuming the other properties of a triangle presentation. The other direction of implication is completely analogous. Assume $6^{\prime}$, and let $(u,t,\sigma(s))$ and $(v,w,\sigma(t))$ be elements of $\mathcal{T}$ with $\text{dim}(u)+\text{dim}(v)+\text{dim}(w)<n$. Then $(t,\sigma(s),u)\in \mathcal{T}$ and $\text{dim}(v)+\text{dim}(w)+\text{dim}(\sigma(s))=n-\text{dim}(u)<n$, thus we can apply the hypothesis to the pair $(v,w,\sigma(t))$ and $(t,\sigma(s),u)$ in $\mathcal{T}$ (with $t$ playing the role of $r$ in the statement) to obtain a unique $r^{\prime}$ such that $(v,r^{\prime},u)$ and $(w, \sigma(s), \sigma(r^{\prime}))$ are both elements of $\mathcal{T}$. Writing $r=\sigma(r^{\prime})$ gives the desired result.
 \end{proof}

\bigskip

\noindent Given a triangle presentation $\mathcal{T}$, we define a group $$\Gamma_{\mathcal{T}}=\langle u\in \Pi\ |\ uvw=1\ \text{if}\ (u,v,w)\in \mathcal{T},\ u\sigma(u)=1\rangle.$$ This group will always admit a type-rotating action on a type $\tilde{\text{A}}_{n-1}$ building whose action on the 1-skeleton is precisely the action on the Cayley graph of $\Gamma_{\mathcal{T}}$. Conversely, any group which acts (in a type rotating way) on a type $\tilde{\text{A}}_{n-1}$ building admits a triangle presentation such that the action on the $1$-skeleton is isomorphic to the action on its Cayley graph. In this sense, triangle presentations give a combinatorial characterization of groups acting simply transitively on the vertices of $\tilde{\text{A}}_{n-1}$ buildings. We summarize this in the statement of the following theorem due to Cartwright, and generalizing the earlier results of \cite{MR1232965} in the case $n=3$.

\begin{thm}\cite[Theorem 2.5]{MR1320514} A group $\Gamma$ admits a type-rotating action on a locally finite thick building of type $\tilde{\text{A}}_{n-1}$ which is simply transitive on the vertices if and only if $\Gamma$ admits a triangle presentation of type $\tilde{\text{A}}_{n-1}$.
\end{thm}

This result motivates the study of triangle presentations, since they provide combinatorial descriptions of the rich geometric structure of affine buildings. Below we include some examples.

\begin{ex}\label{numbertheoryex} \textbf{Cyclic planar difference sets}. We present a generalization of the examples \cite[pp.156-157]{MR1232965}. These may be known to experts, but we could not find them in the literature described with this generality. A \textit{cyclic planar difference set} consists of a natural number $N=q^{2}+q+1$ and a subset $D\subseteq \mathbbm{Z}/N\mathbbm{Z}$ of size $q+1$ such that every non-zero element $m\in  \mathbbm{Z}/N\mathbbm{Z}$ may be written $m=d_{1}-d_{2}$ for a unique pair $(d_{1},d_{2})\in D\times D$. Given a cyclic planar difference set, we can define the structure of an abstract projective plane whose points are the elements of $\mathbbm{Z}/N\mathbbm{Z}$, and the lines are the translates of $D$, with a point incident to a line when it is an element of the set. Cyclic planar difference sets were first studied in \cite{MR1501951}, where they were shown to always exist if $q=p^{n}$ for some prime $p$. General difference sets are of interest in combinatorial design theory.

Given a cyclic planar difference set $D$, any translate $D+s$ is also a cyclic planar difference set. Furthermore, translation by $s$ induces an isomorphism of the associated projective geometries. A \textit{numerical multiplier} of a cyclic planar difference set is an integer $t$ such that $tD=D+s$ for some $s$. If $q=p^{n}$, then $p$ is a numerical multiplier for any cyclic planar difference set \cite[Theorem 4.5]{MR23536}. However, it is desirable for us to find $D$ that are \textit{invariant} under multiplication by $p$, i.e. the associated $s$ is $0$. We call a cyclic planar difference set \textit{standard} if $q=p^{n}$ and $D$ is invariant under multiplication by $p$.

For any cyclic planar difference set $D$ with $q$ a prime power, we claim there is a canonical translate $D_{0}=D+s_{0}$ which is standard. First we claim there is a unique translate $D_{0}=D+s_{0}$ such that $\sum_{d_{i}\in D_{0}} d_{i}=0$. To see this, set $c=\sum_{d_{i}\in D} d_{i}$. Then for any $s$, the sum of the elements in $D+s$ is $c+(q+1)s$. Since $q+1$ is relatively prime to $N=q^{2}+q+1$, it is invertible hence $$s_{0}=-(q+1)^{-1}c\ \text{mod}\ N.$$ Furthermore, multiplication by $p$ preserves the property have having $0$ sum, hence this translate is standard.

Now, given a standard cyclic planar difference set $D$ on $\mathbbm{Z}/N\mathbbm{Z}$ with $q=p^{n}$, we will define an $n=3$ triangle presentation of order $q$. Consider the induced projective geometry with point line correspondence $\sigma$

\begin{align*}
\Pi_{1}&:=\mathbbm{Z}/N\mathbbm{Z}\\
\Pi_{2}&:=\{m+D\}_{m\in \mathbbm{Z}/N\mathbbm{Z}}\\
\sigma(m)&:=m+D,\ \text{for}\ m\in \Pi_{1}\\
\end{align*}

That $\sigma$ is a bijection follows from the definition of planar difference set. Now, define

$$\mathcal{T}^{\prime}:=\{\ (m,\ m+d,\ m+(q+1)d\ )\in \Pi_{1}\times \Pi_{1}\times \Pi_{1}\ : m\in \mathbbm{Z}/N\mathbbm{Z},\ d\in D\}$$

$$\mathcal{T}^{\prime \prime}:=\{\ (\sigma(m_{3}),\ \sigma(m_{2}),\ \sigma(m_1) )\ :\ (m_1,\ m_{2},\ m_{3})\in \mathcal{T}^{\prime}\}$$

$$\mathcal{T}:=\mathcal{T}^{\prime}\cup \mathcal{T}^{\prime \prime}$$

\begin{prop} $\mathcal{T}$ as defined above is a triangle presentation.
\end{prop}

\begin{proof}
Note that in the $n=3$ setting, condition \ref{theweirdone} from Definition \ref{tridefn} is vacuously true, and conditions \ref{lambdareflection} and \ref{mod} follow directly from the definition of $\mathcal{T}$. To verify the remaining conditions, first consider triples in $\mathcal{T}^{\prime}$. Condition \ref{incidence} for this subset follows directly from the definitions. To prove \ref{cyclicinvariance}, consider $(m,\ m+d,\ m+(p+1)d\ ) \in \mathcal{T}^{\prime}$. Then to check $( m+d,\ m+(q+1)d,\ m )\in \mathcal{T}^{\prime}$, note $m+(q+1)d=(m+d)+qd$ and since $D$ is standard $qd\in \mathcal{D}$. This is of the correct form for an entry following $m+d$ in $\mathcal{T}^{\prime}$. Furthermore, $(m+d)+(q+1)qd=m+(q^{2}+q+1)d=m\ \text{mod}\ N $, hence the third entry is also of the correct form. Condition \ref{uniqueness} from Definition \ref{tridefn} is obvious.

Now for triples in $\mathcal{T}^{\prime \prime}$, \ref{incidence} in Definition \ref{tridefn} follows from \ref{incidence} and \ref{cyclicinvariance} for $\mathcal{T}^{\prime}$ triples. Similarly \ref{cyclicinvariance} comes from $\ref{cyclicinvariance}$ for $\mathcal{T}^{\prime}$ triples. It remains to show \ref{uniqueness}. But this is equivalent to showing that if $(u,v,w), (u^{\prime}, v, w)\in \mathcal{T}^{\prime}$, then $u=u^{\prime}$. Suppose $m^{\prime}+d^{\prime}=m + d$ and $m^{\prime}+(q+1)d^{\prime}=m+(q+1){}d$. But since $q$ is invertible in the ring $\mathbbm{Z}/N\mathbbm{Z}$, subtracting the equations yields $d=d^{\prime}$, hence $m=m^{\prime}.$

\end{proof}

We include some examples:

\begin{itemize}
\item
Let $q=4=2^{2}$, so the group in question is $\mathbbm{Z}/21\mathbbm{Z}$. A cyclic planar difference set is given by $D=\{0,1,4, 14, 16\}$, this is not invariant under multiplication by $p=2$, but applying the procedure described above gives us the standard set $D_{0}=D+14=\{14,15,18,7,9\}$
{}
\item
Let $q=7=p$, so the group in question is $\mathbbm{Z}/57\mathbbm{Z}$. A cyclic planar difference set is given by $D=\{0,1,3,13,32,36,43,52\}$. The associated standard set is given by $D_{0}=D+6=\{6,7,9,19,38,42,49,1\}$ 
\end{itemize}


A general family of standard cyclic planar difference sets arises in the context of field extensions. Let $q=p^{n}$ and consider the finite field $\mathbbm{F}_{q^{3}}$ of order $q^{3}$. Since $\mathbbm{F}_{q}\subseteq \mathbbm{F}_{q^{3}}$, $\mathbbm{F}_{q^{3}}$ naturally carries the structure of a 3-dimensional vector space over $\mathbbm{F}_{q}$. The set $\Pi_{1}$ of lines (which are points in the associated projective space) can be identified with the group $\mathbbm{F}^{\times}_{q^{3}}/\mathbbm{F}^{\times}_{q}$, which is a cyclic group of order $q^{2}+q+1$. Furthermore the structure of the extension gives a $\mathbbm{F}_{q}$ valued trace on $\mathbbm{F}_{q^{3}}$. Note that while the trace of a single line is undefined, having $0$ trace is a well-defined notion for one dimensional subspaces. Thus we can define $D:=\{x\in \Pi_{1}\ :\ Tr(x)=0\}$. The cyclic group $\Pi_{1}$ together with $D$ is a standard cyclic planar difference set. 

The corresponding triangle presentations in this case were given in \cite[pp.156-157]{MR1232965}, where in addition it is shown the corresponding building is isomorphic to the Bruhat-Tits building of $\text{PGL}(3, \mathbbm{F}_{q}((t)))$. This class of examples can be generalized to arbitrary $n\ge 3$, showing in particular that there exists triangle presentations for all prime powers $q$ and all $n\ge 3$ associated with groups acting on the Bruhat-Tits building for $\text{PGL}(n, \mathbbm{F}_{q}((t)))$ \cite{MR1613560}.
\end{ex}

\begin{ex}\cite{MR1232966}\label{explicitex} Here we present an example that does not appear to lie in any known infinite family. This example is ``exotic'' in the sense that the corresponding building is not a Bruhat-Tits building. This example also illustrates the smallest parameters for which our construction works, namely $p=2$, $n=3$, $q=3$. A complete classification of triangle presentations with $n=q=3$ is given in \cite{MR1232966}. We present here their triangle presentation with label $15.1$.

Set $\Pi_{1}:=\{p_{i}\}^{12}_{i=0}$ and $\Pi_{2}:=\{l_{i}\}^{12}_{i=0}$. The involution $\sigma$ is given by 

$$\sigma(p_0)=l_{0},\ \sigma(p_1)=l_{3},\ \sigma(p_2)=l_{12},\ \sigma(p_3)=l_{1},\ \sigma(p_4)=l_{9},\ \sigma(p_5)=l_{10},\ \sigma(p_6)=l_{8},$$ 
$$\sigma(p_7)=l_{2},\ \sigma(p_8)=l_{11},\ \sigma(p_9)=l_{6},\ \sigma(p_{10})=l_{4},\ \sigma(p_{11})=l_{5},\ \sigma(p_{12})=l_{7}$$

We now list the triples of $\mathcal{T}$ conatined in $\Pi_{1}\times \Pi_{1}\times \Pi_{1}$. To obtain the triples in $\Pi_{2}\times \Pi_{2}\times \Pi_{2}$ we simple apply $\sigma$ and reverse the order to the list. Also we only include one representative for each cyclic permutation class:

$$(p_0, p_0, p_0),\ (p_{10}, p_{10}, p_{5}),\ (p_{11}, p_{11}, p_{5}),\ (p_{0}, p_{1}, p_{4}),\ (p_{0}, p_{4}, p_{2}),\ (p_{0}, p_{6}, p_{12}),$$

 $$(p_{1}, p_{3}, p_{5}),\ (p_{1}, p_{7}, p_{3}),\  (p_{1}, p_{9}, p_{6}),\ (p_{2}, p_{7}, p_{3}),\ (p_{2}, p_{5}, p_{3}),\ (p_{2}, p_{12}, p_{8}),$$

$$(p_{4}, p_{9}, p_{10}),\ (p_{4}, p_{10}, p_{8}),\ (p_{6}, p_{8}, p_{11}),\ (p_6,p_9,p_7),\ (p_{7}, p_{8}, p_{12}),\ (p_{9}, p_{12}, p_{11})$$

From the above data you can work out the lines as sets of points on the corresponding projective plane. For example, since $\sigma(p_{1})=l_{3}$, $l_{3}$ consists of the set of points immediately following $p_{1}$ in the list, i.e. $l_{3}=\{p_{4},p_{3},p_{7},p_{9}\}$

\end{ex}

\begin{ex}\label{charzero} \textbf{Bruhat-Tits buildings in characteristic $0$}. The previous examples of triangle presentations are either non-linear, or have the Bruhat-Tits building of the local field $\mathbbm{F}_{q}((t))$ as their underlying building. In a different direction \cite{MR2881229} classifies all discrete groups which act transitively on the vertices of a Bruhat-Tits building of dimension at least $4$ over a non-Archimedean local field of characteristic $0$. In particular, they classify all simply transitive actions. It turns out there are not too many, which starkly contrasts with the positive characteristic cases described above. To our knowledge, explicit triangle presentations for the examples in \cite{MR2881229} have not been written down.

\end{ex}


\begin{ex}\label{degenerate} As described above, $\tilde{\text{A}}_{n-1}$ triangle presentations are combinatorial structures defined over a finite projective geometry of (algebraic) dimension n and order $q$. When studying combinatorial structures over projective geometries of order $q$, it is natural to consider this structure in the degenerate case of $q=1$, which is often interpreted in the context of the field with one element $\mathbbm{F}_{1}$ as envisioned by Tits \cite{MR0108765}. A projective geometry over $\mathbbm{F}_{1}$ of algebraic dimension $n$ degenerates to the collection of subsets of a finite set of size $n$, with incidence relation coming from subset inclusion. In this way, combinatorial structures defined on finite projective geometries of order $q$ often degenerate to familiar combinatorial structures on finite sets, providing motivation and intuition for the general structure. In this spirit, we consider the following \textit{degenerate} triangle presentation.

Fix a natural number $n$ and set $X=\{0,1,2,3,\dots, n\}$. Let $\Pi_{k}$ denote the collection of subsets of $X$ of size $k$, and set $\Pi:=\bigcup^{n}_{k=1} \Pi_{k}$, the set of proper subsets. For $x\in \Pi$, we denote by $\text{dim}(x)$ the cardinality of $x$. We have an \textit{incidence relation} $x\sim y$ if either $x\subseteq y$ or $y\subseteq x$. We also have an involution $\sigma: \Pi\rightarrow \Pi$, $\sigma(\Pi_{k})=\Pi_{n+1-k}$ which takes a subset of $X$ to its complement.

Now consider the following subsets of $ \Pi\times \Pi\times \Pi$, 
$$\mathcal{T}^{\prime}=\{(x,y,z)\in \Pi\times \Pi\times \Pi\ :\ x,y,z\ \text{are pairwise disjoint, and}\ x\cup y\cup z=X\}$$

$$\mathcal{T}^{\prime \prime}=\{(x,y,z)\in \Pi\times \Pi\times \Pi\ :\ (\sigma(z),\sigma(y),\sigma(x))\in \mathcal{T}^{\prime} \}$$ 

$$\mathcal{T}=\mathcal{T}^{\prime}\cup \mathcal{T}^{\prime \prime} $$

Then it is easy to verify that $\mathcal{T}$ satisfies all the conditions of a triangle presentation, except that the incidence relation on $\Pi$ does not give a projective geometry (it gives instead a \textit{degenerate} projective geometry). We will see that nevertheless, the construction of our fiber functor from Theorem \ref{TriangleSLnThm} does apply to these degenerate triangle presentations as long as $\mathbbm{k}$ has characteristic $0$, as expected from the $\mathbbm{F}_{1}$ philosophy.

\end{ex}

\section{Categories and diagrammatics}\label{categories}

In this section, we review some categories that have descriptions in terms of diagrammatic generators and relations. This type of description is in the spirit of skein theory for planar algebras and knot polynomials \cite{math.QA/9909027}. For a general reference on linear monoidal categories, we refer the reader to \cite{MR3242743}.

\subsection{Polynomial $\text{GL}_{n}$ webs}

In this section we describe a linear monoidal category called the \textit{polynomial web category for} $\text{GL}_{n}$ over a field $\mathbbm{k}$ \cite{MR3263166}, \cite{MR4135417}. Our presentation closely follows \cite{MR4135417}. Indeed, in the description below, our presentation is easily derived as the quotient of $\text{Web}$ \cite[Definition 4.7]{MR4135417} by the monoidal ideal generated by the identities of objects of weight $m$, where $m>n$. If $\mathbbm{k}$ is algebraically closed, the monoidal category described below is equivalent to the monoidal category of \textit{polynomial representations} of $\text{Gl}_{n}$, which can be equivalently described as the full subcategory of tilting modules for $\text{GL}_{n}$ generated by wedge powers of the standard $n$-dimensional representation (see \cite[Remark 4.15]{MR4135417}).

Objects in our monoidal category will be sequences of numbers from the set $\{1, \cdots, n\}$. The monoidal product (which we sometimes denote $\otimes$) is concatenation of sequences, and the monoidal unit consists of the empty sequence. To describe morphisms, we introduce the generating morphisms 

$$
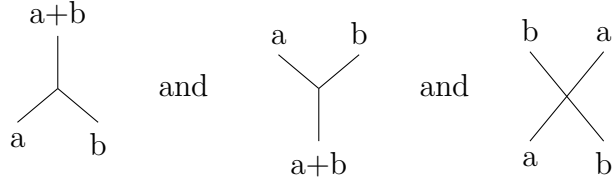
\begin{tikzpicture}[scale=0.7, baseline = -.1cm]
\draw (90:1) node[above] {a+b}  -- (0,0);
\draw (220:1) node[below] {a} -- (0,0);
\draw (320:1) node[below] {b}  -- (0,0);
\end{tikzpicture}\ \ \ \  \text{and}\ \ \ \ \ 
\begin{tikzpicture}[scale=0.7, baseline = -.1cm]
\draw (40:1) node[above] {b} -- (0,0);
\draw (140:1) node[above] {a} -- (0,0);
\draw (270:1)  node[below] {a+b} -- (0,0);
\end{tikzpicture}\ \ \ \ \text{and}\ \ \ \ 
\begin{tikzpicture}[scale=0.7, baseline = -.1cm]
\draw  (-0.7,-1) node[below] {a} -- (0.7,0.7) node[above] {a};
\draw  ( 0.7,-1)node[below] {b} -- (-0.7,0.7) node[above] {b};
\end{tikzpicture}
$$

\noindent which we call merges (from $(a,b)\rightarrow (a+b)$), splits (from $(a+b)\rightarrow (a,b)$), and crossings (from $(a,b)\rightarrow (b,a)$) respectively  (note we read diagrams bottom to top). The labels of all strings in generating morphisms must lie in the set $\{1,\dots, n\}$. For example if $a+b>n$, there is no allowed label $a+b$, hence no merge or split morphism of that type, though crossings will exist. Morphisms will then be formal $\mathbbm{k}$-linear combinations of vertical and horizontal compositions of these generators, subject to the following linear relations (cf. \cite[Definition 4.7]{MR4135417}):

\begin{equation}\label{associativity}
\begin{tikzpicture}[scale=0.7, baseline = 0 cm]
\draw (-1,1) node[above]{a} -- (-0.5,0.5);
\draw (-0.5,0.5) --  (0,0);
\draw (0,1)node[above]{b}--(-0.5,0.5);
\draw (1,1)node[above]{c}--(0,0);
\draw (0,0)--(0,-0.5)node[below]{a+b+c};
\end{tikzpicture}\ \ \ =\ \ \ 
\begin{tikzpicture}[scale=0.7, baseline = 0cm]
\draw (-1,1) node[above]{a} -- (0,0) ;
\draw (0.5,0.5) -- (0,0);
\draw (0,1)node[above]{b}--(0.5,0.5);
\draw (1,1)node[above]{c}--(0.5,0.5);
\draw (0,0)--(0,-0.5)node[below]{a+b+c};
\end{tikzpicture}\ \ \ \  ,\ \ \ \ \ \ \ \ \ 
\begin{tikzpicture}[scale=0.7, baseline = -.3 cm]
\draw (-1,-1) node[below]{a} -- (-0.5,-0.5);
\draw (-0.5,-0.5) --  (0,0);
\draw (0,-1)node[below]{b}--(-0.5,-0.5);
\draw (1,-1)node[below]{c}--(0,0);
\draw (0,0)--(0,0.5)node[above]{a+b+c};
\end{tikzpicture}\ \ \ =\ \ \ 
\begin{tikzpicture}[scale=0.7, baseline = -.3cm]
\draw (-1,-1) node[below]{a} -- (0,0) ;
\draw (0.5,-0.5) -- (0,0);
\draw (0,-1)node[below]{b}--(0.5,-0.5);
\draw (1,-1)node[below]{c}--(0.5,-0.5);
\draw (0,0)--(0,0.5)node[above]{a+b+c};
\end{tikzpicture}
\end{equation}

\begin{equation}\label{burstingbigons}
\begin{tikzpicture}[scale=0.7, baseline = -.1cm]
\draw (0,-1)node[below]{a+b} -- (0,-0.5) ;
\draw (270:0.5) arc (270:90:0.5);
\draw (-90:0.5) arc (-90:90:0.5);
\draw (0,0.5)--(0,1) node[above]{a+b};
\node(none) at (1,0) {a};
\node(none) at (-1,0) {b};
\end{tikzpicture}\ \ \ = \ \ \ \  \left({a+b \atop a}\right)\ \  \begin{tikzpicture}[scale=1, baseline = -.1cm]
\draw (0,-1) -- (0,1) ;
\end{tikzpicture}
\end{equation}

\begin{equation}\label{bialgebra}
\begin{tikzpicture}[scale=0.7, baseline = -.1cm]
\draw (-0.5,-1) node[below]{a} -- (0,-0.5)  ;
\draw (0.5,-1) node[below]{c} -- (0,-0.5) ;
\draw (0,-0.5) -- (0,0.5) ;
\draw (0,0.5) -- (-0.5,1) node[above]{b} ;
\draw (0,0.5) -- (0.5,1) node[above]{d} ;
\end{tikzpicture}\ =\ 
\sum_{s} \ \ \ \begin{tikzpicture}[scale=0.7, baseline = -.1cm]
\draw (-0.5,-1) node[below]{a} -- (-0.5,1) node[above]{b} ;
\draw (0.5,-1) node[below]{c} -- (0.5,1) node[above]{d}  ;
\draw (-0.5,-0.5) -- (0.5,0.5) ;
\draw (0.5,-0.5) -- (-0.5,0.5) ;
\node(none) at (-0.8,0) {s};
\end{tikzpicture}
\end{equation}

We sometimes do not label strings, but in this case unlabeled strings can be deduced uniquely from the given labels. Also summations aways occur over all allowable values of the corresponding label (we allow the ``empty string'' as a string with label $0$ in summations). We call the relations in equations \ref{associativity} \textit{coassociativity} and \textit{associativity} respectively. We call the relations in equation \ref{burstingbigons} the \textit{bigon bursting} relation, and the relation from equation \ref{bialgebra} the \textit{bialgebra} relations. Indeed, it follows from the relations that the crossing generator makes $\text{PolyWeb}(\text{GL}_{n})$ into a symmetric monoidal category. Setting $A=\bigoplus^{n}_{a=0} a$ in the additive envelope, then the trivalent vertices give $A$ the structure of a coalgebra and an algebra. Relation \ref{bialgebra} is precisely the statement that this is a bialgebra internal to $\text{PolyWeb}(\text{GL}_{n})$.

There are many additional relations that follow as consequences of the above. In particular, the crossing generator is actually redundant. Indeed, a consequence of the above relations is that the braid can be written in terms of the trivalent vertices \cite[Equation 4.36]{MR4135417}

\begin{equation}\label{braiddefine}
\begin{tikzpicture}[scale=0.7, baseline = -.1cm]
\draw  (-0.5,-1)  -- (0.5,1); 
\draw  ( 0.5,-1) -- (-0.5,1);
\node(none) at (-0.9,-1.1) {a};
\node(none) at (0.9,-1.1) {b};
\node(none) at (0.9,1.1) {a};
\node(none) at (-0.9,1.1) {b};
\end{tikzpicture}= \sum_{t} \ \ \ (-1)^{t}\ \ \  \begin{tikzpicture}[scale=0.7, baseline = -.1cm]
\draw (-0.5,-1)  -- (-0.5,-0.5) ;
\draw (-0.5,-0.5) -- (0.5,-0.2) ;
\draw (-0.5,-0.5) -- (-0.5,0.5);
\draw (-0.5,0.5) -- (-0.5,1) (-0.5,1) ;
\draw (0.5,0.2) -- (-0.5,0.5);
\draw (0.5,-1)  -- (0.5,-0.2);
\draw (0.5,-0.2)--(0.5,0.5);
\draw (0.5,0.5)--(0.5,1) (0.5,1) ;
\node(none) at (-0.7,0) {t};
\node(none) at (-0.9,-1.1) {a};
\node(none) at (0.9,-1.1) {b};
\node(none) at (0.9,1.1) {a};
\node(none) at (-0.9,1.1) {b};
\end{tikzpicture}
\end{equation}

We can use this to given a presentation for the category $\text{PolyWeb}(\text{GL}_{n})$ purely in terms relations between trivalent vertices. In particular, we have the following \textit{square-switch relations}:

\begin{equation}\label{squareswitch1}
 \begin{tikzpicture}[scale=0.9, baseline = -.1cm]
\draw (-0.5,-1) node[below] {a} -- (-0.5,-0.5) ;
\draw (-0.5,-0.5) -- (0.5,-0.2) ;
\draw (-0.5,-0.5) -- (-0.5,0.5);
\draw (-0.5,0.5) -- (-0.5,1) (-0.5,1);
\draw (0.5,0.2) -- (-0.5,0.5);
\draw (0.5,-1) node[below] {b} -- (0.5,-0.2);
\draw (0.5,-0.2)--(0.5,0.5);
\draw (0.5,0.5)--(0.5,1) (0.5,1) ;
\node(none) at (0,0.7) {c};
\node(none) at (0,-0.7) {d};
\end{tikzpicture}=\ \ \sum_{t}\ \ {a-b+c-d\choose t}\ \ \begin{tikzpicture}[scale=0.9, baseline = -.1cm]
\draw (-0.5,-1) node[below]{a}--(-0.5,1);
\draw (0.5,-1) node[below]{b}--(0.5,1);
\draw (0.5,-0.5)--(-0.5,-0.2);
\draw (-0.5,0.2)--(0.5,0.5);
\node(none) at (0,0.7) {d-t};
\node(none) at (0,-0.7) {c-t};
\end{tikzpicture}
\end{equation}

\bigskip

\begin{equation}\label{squareswitch2}
\begin{tikzpicture}[scale=0.9, baseline = -.1cm]
\draw (-0.5,-1) node[below]{a}--(-0.5,1);
\draw (0.5,-1) node[below]{b}--(0.5,1);
\draw (0.5,-0.5)--(-0.5,-0.2);
\draw (-0.5,0.2)--(0.5,0.5);
\node(none) at (0,0.7) {d};
\node(none) at (0,-0.7) {c};
\end{tikzpicture}=\ \ \sum_{t}\ \ {b-a+d-c\choose t}\ \ \begin{tikzpicture}[scale=0.9, baseline = -.1cm]
\draw (-0.5,-1) node[below] {a} -- (-0.5,-0.5) ;
\draw (-0.5,-0.5) -- (0.5,-0.2) ;
\draw (-0.5,-0.5) -- (-0.5,0.5);
\draw (-0.5,0.5) -- (-0.5,1) (-0.5,1);
\draw (0.5,0.2) -- (-0.5,0.5);
\draw (0.5,-1) node[below] {b} -- (0.5,-0.2);
\draw (0.5,-0.2)--(0.5,0.5);
\draw (0.5,0.5)--(0.5,1) (0.5,1) ;
\node(none) at (0,0.7) {c-t};
\node(none) at (0,-0.7) {d-t};
\end{tikzpicture}
\end{equation}

\begin{prop}\cite[Appendix A]{MR4135417} Consider a monoidal category generated by only the trivalent merges and splits, with crossing defined as in equation \ref{braiddefine} satisfying relations \ref{associativity} and \ref{burstingbigons}. Then relation \ref{bialgebra} is satisfied if and only if both square switch relations \ref{squareswitch1} and \ref{squareswitch2} are satisfied.
\end{prop}

\begin{rem}\label{standardfunctorpolyweb}  There is a standard monoidal functor $F:\text{PolyWeb}(\text{GL}_{n})\rightarrow \text{Vec}$, described in \cite[Theorem 4.14]{MR4135417}. Let $V=\mathbbm{k}^{n}$. Then the generating object $a\in \text{PolyWeb}(\text{GL}_{n})$ gets sent to $\bigwedge^{a} V$. The merging trivalent vertex $\begin{tikzpicture}[scale=0.4, baseline = -.1cm]
\draw (90:1) node[above] {a+b}  -- (0,0);
\draw (220:1) node[below] {a} -- (0,0);
\draw (320:1) node[below] {b}  -- (0,0);
\end{tikzpicture}$ maps to the canonical surjection $ (\bigwedge^{a} V )\otimes (\bigwedge^{b} V )\rightarrow \bigwedge^{a+b}V$ defined by  

$$v_{k_1}\wedge \dots \wedge v_{k_{a}}\otimes v_{j_1}\wedge \dots \wedge v_{j_b}\mapsto v_{k_1}\wedge \dots \wedge v_{k_{a}}\wedge v_{j_1}\wedge \dots \wedge v_{j_b}$$

\noindent while the split vertex $\begin{tikzpicture}[scale=0.4, baseline = -.1cm]
\draw (40:1) node[above] {b} -- (0,0);
\draw (140:1) node[above] {a} -- (0,0);
\draw (270:1)  node[below] {a+b} -- (0,0);
\end{tikzpicture}$ is assigned the linear map $\bigwedge^{a+b} V\rightarrow \bigwedge^{a}V \otimes \bigwedge^{b} V$ defined by $$v_{k_1}\wedge \dots\ v_{k_{a+b}}\mapsto \sum_{g\in (S_{a+b}/S_{a}\times S_{b})_{min}} (-1)^{l(g)} v_{k_{g(1)}} \wedge\ \dots \wedge v_{k_{g(a)}}\otimes v_{k_{g(a+1)}} \wedge\ \dots \wedge v_{k_{g(a+b)}}  $$

\noindent where $(S_{a+b}/S_{a}\times S_{b})_{min}$ denotes a choice of minimal length representatives of left cosets $S_{a+b}/S_{a}\times S_{b}$ and $l(g)$ denotes the length of an element (with respect to the usual Coxeter presentation of the symmetric groups). Under these identifications, the $(a,b)\rightarrow (b,a)$ crossing is assigned to the isomorphism $ (\bigwedge^{a} V )\otimes (\bigwedge^{b} V )\rightarrow (\bigwedge^{b} V )\otimes (\bigwedge^{a} V )$ given by $$v\otimes w\mapsto (-1)^{ab} w\otimes v.$$ Notice this is \textit{not} the ordinary braiding on $\text{Rep}(\text{GL}_{n}(\mathbbm{k}))$ which simply permutes tensor factors, but rather is the symmetric braiding associated to the involutive central element $ -\text{Id}_{n}\in Z(\text{GL}_{n}(\mathbbm{k}))$.
\end{rem}

The following proposition is a combination of a remark in \cite[pp. 8]{MR3263166} and \cite[Proposition 1]{1808.10575}, but we include a proof here since we are in a positive characteristic setting.

\begin{prop}\label{easytohardsquareswitch} For $1\le a,b,\le n$, if $\text{char}(\mathbbm{k})>c,d$, then relations \ref{squareswitch1} and \ref{squareswitch2} follow from \ref{associativity},\ref{burstingbigons} and the special case of square switch relations
\end{prop}

\begin{equation}\label{squareswitcha11}\begin{tikzpicture}[scale=0.7, baseline = -.1cm]
\draw (-0.5,-1.2) -- (-0.5,-0.6) ;
\draw (-0.5,-0.6) -- (0.5,-0.4) ;
\draw (-0.5,-0.6) -- (-0.5,0.3);
\draw (-0.5,0.3) -- (-0.5,1.2);
\draw (0.5,0.4) -- (-0.5,0.6);
\draw (0.5,-1.2)--(0.5,-0.2);
\draw (0.5,-0.2)--(0.5,0.6);
\draw (0.5,0.6)--(0.5,1.2);
\node(none) at (-1,0) {a-1};
\node(none) at (1,0) {2};
\node(none) at (-0.9,1.2) {a};
\node(none) at (0.9,1.2) {1};
\node(none) at (-0.9,-1.2) {a};
\node(none) at (0.8,-1.2) {1};
\end{tikzpicture}\ =\ 
\begin{tikzpicture}[scale=0.7, baseline = -.1cm]
\draw (-0.5,-1.2) -- (0,-0.6) ;
\draw (0,-0.6) -- (0,0.6) ;
\draw (0,0.6) -- (-0.5,1.2);
\draw (0,0.6) -- (0.5,1.2);
\draw (0.5,-1.2)--(0,-0.6);
\node(none) at (0.7,0) {a+1};
\node(none) at (-0.9,1.2) {a};
\node(none) at (0.9,1.2) {1};
\node(none) at (-0.9,-1.2) {a};
\node(none) at (0.8,-1.2) {1};
\end{tikzpicture}\ +\ 
(a-1)\ \begin{tikzpicture}[scale=0.7, baseline = -.1cm]
\draw (-0.5,-1.2) -- (-0.5,1.2) ;
\draw (0.5,-1.2) -- (0.5, 1.2) ;
\node(none) at (-0.9,-1.2) {a};
\node(none) at (0.8,-1.2) {1};
\end{tikzpicture}
\end{equation}

\begin{equation}\label{squareswitch1aa}
\begin{tikzpicture}[scale=0.7, baseline = -.1cm]
\draw (-0.5,-1.2) -- (-0.5,-0.6) ;
\draw (0.5,-0.6) -- (-0.5,-0.4) ;
\draw (-0.5,-0.6) -- (-0.5,0.3);
\draw (-0.5,0.3) -- (-0.5,1.2);
\draw (-0.5,0.4) -- (0.5,0.6);
\draw (0.5,-1.2)--(0.5,-0.2);
\draw (0.5,-0.2)--(0.5,0.6);
\draw (0.5,0.6)--(0.5,1.2);
\node(none) at (-1,0) {2};
\node(none) at (1,0) {a-1};
\node(none) at (-0.9,1.2) {1};
\node(none) at (0.9,1.2) {a};
\node(none) at (-0.9,-1.2) {1};
\node(none) at (0.8,-1.2) {a};
\end{tikzpicture}\ =
\begin{tikzpicture}[scale=0.7, baseline = -.1cm]
\draw (-0.5,-1.2) -- (0,-0.6) ;
\draw (0,-0.6) -- (0,0.6) ;
\draw (0,0.6) -- (-0.5,1.2);
\draw (0,0.6) -- (0.5,1.2);
\draw (0.5,-1.2)--(0,-0.6);
\node(none) at (0.7,0) {a+1};
\node(none) at (-0.9,1.2) {1};
\node(none) at (0.9,1.2) {a};
\node(none) at (-0.9,-1.2) {1};
\node(none) at (0.8,-1.2) {a};
\end{tikzpicture}\ +\ 
(a-1)\ \begin{tikzpicture}[scale=0.7, baseline = -.1cm]
\draw (-0.5,-1.2) -- (-0.5,1.2) ;
\draw (0.5,-1.2) -- (0.5, 1.2) ;
\node(none) at (-0.9,-1.2) {1};
\node(none) at (0.8,-1.2) {a};
\end{tikzpicture}
\end{equation}

\begin{proof}

We can use \ref{squareswitcha11} and \ref{squareswitch1aa} to reduce the following diagram in two different ways $$\begin{tikzpicture}[scale=0.7, baseline = -.1cm]
\draw (-1,-1.2) -- (-1,1.2);
\draw (-1,-0.6) -- (0,-0.4);
\draw (1,-0.6) -- (0,-0.4);
\draw (0,-0.4) -- (0,0.4);
\draw (0,0.4)-- (-1,0.6);
\draw (0,0.4)-- (1,0.6);
\draw(1,-1.2)--(1,1.2);
\node(none) at (-1.5,-1.2) {a};
\node(none) at (1.5,-1.2) {b};
\node(none) at (-1.5,1.2) {a};
\node(none) at (1.5,1.2) {b};
\node(none) at (-0.3,0) {2};
\node(none) at (-0.5,-0.9) {1};
\node(none) at (0.5,-0.9) {1};
\node(none) at (-0.5,0.9) {1};
\node(none) at (0.5,0.9) {1};
\node(none) at (-1.7,0) {a-1};
\node(none) at (1.7,0) {b-1};
\end{tikzpicture}$$

Equating both sides gives us the special case of the square switch relation

\begin{equation}
\begin{tikzpicture}[scale=0.7, baseline = -.1cm]
\draw (-0.5,-1.2) -- (-0.5,-0.6) ;
\draw (-0.5,-0.6) -- (0.5,-0.4) ;
\draw (-0.5,-0.6) -- (-0.5,0.3);
\draw (-0.5,0.3) -- (-0.5,1.2);
\draw (0.5,0.4) -- (-0.5,0.6);
\draw (0.5,-1.2)--(0.5,-0.2);
\draw (0.5,-0.2)--(0.5,0.6);
\draw (0.5,0.6)--(0.5,1.2);
\node(none) at (-1.1,0) {a-1};
\node(none) at (1.1,0) {b+1};
\node(none) at (-0.9,1.2) {a};
\node(none) at (0.9,1.2) {b};
\node(none) at (-0.9,-1.2) {a};
\node(none) at (0.8,-1.2) {b};
\node(none) at (0,-0.9) {1};
\node(none) at (0,0.9) {1};
\end{tikzpicture}\ =\ 
\begin{tikzpicture}[scale=0.7, baseline = -.1cm]
\draw (-0.5,-1.2) -- (-0.5,-0.6) ;
\draw (0.5,-0.6) -- (-0.5,-0.4) ;
\draw (-0.5,-0.6) -- (-0.5,0.3);
\draw (-0.5,0.3) -- (-0.5,1.2);
\draw (-0.5,0.4) -- (0.5,0.6);
\draw (0.5,-1.2)--(0.5,-0.2);
\draw (0.5,-0.2)--(0.5,0.6);
\draw (0.5,0.6)--(0.5,1.2);
\node(none) at (-1.1,0) {a+1};
\node(none) at (1,0) {b-1};
\node(none) at (-0.9,1.2) {a};
\node(none) at (0.9,1.2) {b};
\node(none) at (-0.9,-1.2) {a};
\node(none) at (0.8,-1.2) {b};
\node(none) at (0,-0.9) {1};
\node(none) at (0,0.9) {1};
\end{tikzpicture}\ +\ 
(a-b)\ \begin{tikzpicture}[scale=0.7, baseline = -.1cm]
\draw (-0.5,-1.2) -- (-0.5,1.2) ;
\draw (0.5,-1.2) -- (0.5, 1.2) ;
\node(none) at (-0.9,-1.2) {a};
\node(none) at (0.8,-1.2) {b};
\end{tikzpicture}
\end{equation}

Now suppose we have \ref{squareswitch1} with $c=1$, $d=s$ and $s+1<n$. Then using \ref{burstingbigons} and \ref{associativity}, we have 

\begin{equation}
\begin{tikzpicture}[scale=0.7, baseline = -.1cm]
\draw (-0.5,-1.2) -- (-0.5,1.2) ;
\draw (-0.5,-0.6) -- (0.5,-0.4) ;
\draw (-0.5,0.4) -- (0.5,0.6);
\draw (0.5,-1.2)--(0.5,1.2);
\node(none) at (-0.9,-1.2) {a};
\node(none) at (0.8,-1.2) {b};
\node(none) at (0,-0.9) {1};
\node(none) at (0,0.9) {s};
\end{tikzpicture}\ =\ (s+1)\ \begin{tikzpicture}[scale=0.7, baseline = -.1cm]
\draw (-0.7,-1.2) -- (-0.7,1.2) ;
\draw (-0.7,-0.1) -- (0.7,0.1) ;
\draw (0.7,-1.2)--(0.7,1.2);
\node(none) at (-1.1,-1.2) {a};
\node(none) at (1.1,-1.2) {b};
\node(none) at (0,0.4) {s+1};
\end{tikzpicture}
\end{equation}

Then applying \ref{squareswitch1} we have 
\begin{align*}
\begin{tikzpicture}[scale=0.7, baseline = -.1cm]
\draw (-0.7,-1.2) -- (-0.7,1.2) ;
\draw (-0.7,-0.6) -- (0.7,-0.4) ;
\draw (0.7,0.4) -- (-0.7,0.6) ;
\draw (0.7,-1.2)--(0.7,1.2);
\node(none) at (-1.1,-1.2) {a};
\node(none) at (1.1,-1.2) {b};
\node(none) at (0,-0.9) {s+1};
\node(none) at (0,0.9) {1};
\end{tikzpicture} &=\ \frac{1}{s+1} \begin{tikzpicture}[scale=0.7, baseline = -.1cm]
\draw (-0.7,-1.2) -- (-0.7,1.2) ;
\draw (-0.7,-0.8) -- (0.7,-0.6) ;
\draw (-0.7,-0.3) -- (0.7,-0.1) ;
\draw (0.7,0.4) -- (-0.7,0.6) ;
\draw (0.7,-1.2)--(0.7,1.2);
\node(none) at (-1.1,-1.2) {a};
\node(none) at (1.1,-1.2) {b};
\node(none) at (0,-1) {1};
\node(none) at (0,0.1) {s};
\node(none) at (0,0.9) {1};
\end{tikzpicture}\\
&=  \frac{1}{s+1} \left[\begin{tikzpicture}[scale=0.7, baseline = -.1cm]
\draw (-0.7,-1.2) -- (-0.7,1.2) ;
\draw (-0.7,-0.8) -- (0.7,-0.6) ;
\draw (0.7,-0.3) -- (-0.7,-0.1) ;
\draw (-0.7,0.4) -- (0.7,0.6) ;
\draw (0.7,-1.2)--(0.7,1.2);
\node(none) at (-1.1,-1.2) {a};
\node(none) at (1.1,-1.2) {b};
\node(none) at (0,-1) {1};
\node(none) at (0,0.1) {1};
\node(none) at (0,0.9) {s};
\end{tikzpicture}+(a-b-1-s)\begin{tikzpicture}[scale=0.7, baseline = -.1cm]
\draw (-0.7,-1.2) -- (-0.7,1.2) ;
\draw (-0.7,-0.6) -- (0.7,-0.4) ;
\draw (-0.7,0.2) -- (0.7,0.4) ;
\draw (0.7,-1.2)--(0.7,1.2);
\node(none) at (-1.1,-1.2) {a};
\node(none) at (1.1,-1.2) {b};
\node(none) at (0,-0.9) {1};
\node(none) at (0,0.8) {s-1};
\end{tikzpicture} \right]\\
&= \frac{1}{s+1}\begin{tikzpicture}[scale=0.7, baseline = -.1cm]
\draw (-0.7,-1.2) -- (-0.7,1.2) ;
\draw (0.7,-0.8) -- (-0.7,-0.6) ;
\draw (-0.7,-0.3) -- (0.7,-0.1) ;
\draw (-0.7,0.4) -- (0.7,0.6) ;
\draw (0.7,-1.2)--(0.7,1.2);
\node(none) at (-1.1,-1.2) {a};
\node(none) at (1.1,-1.2) {b};
\node(none) at (0,-1) {1};
\node(none) at (0,0.1) {1};
\node(none) at (0,0.9) {s};
\end{tikzpicture}+\left(\frac{a-b}{s+1}+\frac{s(a-b-1-s)}{s+1} \right) \begin{tikzpicture}[scale=0.7, baseline = -.1cm]
\draw (-0.7,-1.2) -- (-0.7,1.2) ;
\draw (-0.7,-0.1) -- (0.7,0.1) ;
\draw (0.7,-1.2)--(0.7,1.2);
\node(none) at (-1.1,-1.2) {a};
\node(none) at (1.1,-1.2) {b};
\node(none) at (0,0.4) {s};
\end{tikzpicture}\\
&= \begin{tikzpicture}[scale=0.7, baseline = -.1cm]
\draw (-0.7,-1.2) -- (-0.7,1.2) ;
\draw (0.7,-0.6) -- (-0.7,-0.4) ;
\draw (-0.7,0.4) -- (0.7,0.6) ;
\draw (0.7,-1.2)--(0.7,1.2);
\node(none) at (-1.1,-1.2) {a};
\node(none) at (1.1,-1.2) {b};
\node(none) at (0,-0.9) {1};
\node(none) at (0,0.9) {s+1};
\end{tikzpicture}+ (a-b-s) \begin{tikzpicture}[scale=0.7, baseline = -.1cm]
\draw (-0.7,-1.2) -- (-0.7,1.2) ;
\draw (-0.7,-0.1) -- (0.7,0.1) ;
\draw (0.7,-1.2)--(0.7,1.2);
\node(none) at (-1.1,-1.2) {a};
\node(none) at (1.1,-1.2) {b};
\node(none) at (0,0.4) {s};
\end{tikzpicture}
\end{align*}

\noindent By induction, this gives us the square switch relation from Equation \ref{squareswitch1} in the case for $c=1$, $d<n$ satisfying $d<\text{char}(\mathbbm{k})$. An analogous argument gives us \ref{squareswitch1} for arbitrary $c<n, c<\text{char}(\mathbbm{k})$. A similar argument with mirror images gives Equation \ref{squareswitch2}.

\end{proof}

\subsection{$\text{SL}_{n}$ quotients}\label{SLnQuotients}

Given a monoidal category $\mathcal{C}$, a \textit{quotient} is a dominant monoidal functor $\mathcal{C}\rightarrow \mathcal{D}$ to some other monoidal category $\mathcal{D}$. A quotient can thus be obtained by adding new morphisms to the category $\mathcal{C}$. The category of representations of $\text{SL}_{n}$ is closely related to the category of representations of $\text{GL}_{n}$. The obvious difference is that the determinant representation of $\text{GL}_{n}$ restricts to the trivial representation of $\text{SL}_{n}$. Therefore we call all categories we obtain from adding an isomorphism from the object $n$ to the empty object in $\text{PolyWeb}(\text{GL}_{n})$ \textit{$\text{SL}_{n}$ quotients}. 

To construct $\text{SL}_{n}$ quotients, we will consider the category $\text{PolyWeb}(\text{GL}_{n})$, and add two new generators and some relations. In particular, we add morphisms 

$$
\begin{tikzpicture}[scale=0.5, baseline = -.1cm]
\draw (0,-1) node[below]{n}--(0,0);
\mydotw{(0,0)};
\end{tikzpicture}\ \ \ \text{and}\ \ \ \ 
\begin{tikzpicture}[scale=0.5, baseline = -.1cm]
\draw (0,-1) --(0,0)node[above]{n};
\mydotw{(0,-1)};
\end{tikzpicture},
$$

satisfying

\begin{equation}\label{univalent}
\begin{tikzpicture}[scale=0.5, baseline = -.1cm]
\draw (0,-1) node[below]{n} --(0,1);
\end{tikzpicture}=\begin{tikzpicture}[scale=0.5, baseline = -.1cm] 
\draw (0,-1) node[below]{n}--(0,-0.3);
\draw (0,0.3)--(0,1);{}
\mydotw{(0,-0.3)};
\mydotw{(0,0.3)};
\end{tikzpicture}\ \ \ \ \ \ ,\ \ \ \ \ \ \ 
\begin{tikzpicture}[scale=0.5, baseline = -.1cm] 
\draw (0,-0.5)--(0,0.5);
\mydotw{(0,-0.5)};
\mydotw{(0,0.5)};
\end{tikzpicture}=\ \ \ \ \begin{tikzpicture}[scale=0.5, baseline=-.1cm]
\draw[dotted] (0,0) circle (0.8);
\end{tikzpicture}
\end{equation}

\noindent (the dotted circle simply represents the empty diagram, and is included to avoid confusion).

We also would like a compatibility with the crossing generator. There are two sets of relations we consider:

\begin{equation}\label{SL+}
\text{SL}^{+}_{n}\ \text{relations}:\ \ \ \ \ \ \ \ \ \ \begin{tikzpicture}[scale=0.5, baseline = -.1cm]
\draw  (-0.5,-1) node[below] {a} -- (0.5,1);
\draw  ( 0.5,-1)node[below] {n} -- (-0.5,0.5) ;
\mydotw{(-0.5,0.5)};
\end{tikzpicture}=\ \  (-1)^{an}\ \ \begin{tikzpicture}[scale=0.5, baseline = -.1cm]
\draw  (-0.5,-1) node[below] {a} -- (0.5,1);
\draw  ( 0.5,-1)node[below] {n} -- (0.5,-0.5) ;
\mydotw{(0.5,-0.5)};
\end{tikzpicture}\ \ \ \ , \ \ \ \ \begin{tikzpicture}[scale=0.5, baseline = -.1cm]
\draw  (-0.5,-1) node[below] {n} -- (0.5,0.5);
\draw  ( 0.5,-1)node[below] {a} -- (-0.5,1) ;
\mydotw{(0.5,0.5)};
\end{tikzpicture}= \ \ (-1)^{an}\ \  \begin{tikzpicture}[scale=0.5, baseline = -.1cm]
\draw  (-0.5,-1) node[below] {n} -- (-0.5,-0.5);
\draw  ( 0.5,-1)node[below] {a} -- (-0.5,1) ;
\mydotw{(-0.5,-0.5)};
\end{tikzpicture}
\end{equation}

\begin{equation}\label{SL-}
\text{SL}^{-}_{n}\ \text{relations}:\ \ \ \ \ \ \ \ \ \ \begin{tikzpicture}[scale=0.5, baseline = -.1cm]
\draw  (-0.5,-1) node[below] {a} -- (0.5,1);
\draw  ( 0.5,-1)node[below] {n} -- (-0.5,0.5) ;
\mydotw{(-0.5,0.5)};
\end{tikzpicture}=\ \ (-1)^{a}\ \ \begin{tikzpicture}[scale=0.5, baseline = -.1cm]
\draw  (-0.5,-1) node[below] {a} -- (0.5,1);
\draw  ( 0.5,-1)node[below] {n} -- (0.5,-0.5) ;
\mydotw{(0.5,-0.5)};
\end{tikzpicture}\ \ \ \ , \ \ \ \ \begin{tikzpicture}[scale=0.5, baseline = -.1cm]
\draw  (-0.5,-1) node[below] {n} -- (0.5,0.5);
\draw  ( 0.5,-1)node[below] {a} -- (-0.5,1) ;
\mydotw{(0.5,0.5)};
\end{tikzpicture}= \ \ (-1)^{a}\ \ \begin{tikzpicture}[scale=0.5, baseline = -.1cm]
\draw  (-0.5,-1) node[below] {n} -- (-0.5,-0.5);
\draw  ( 0.5,-1)node[below] {a} -- (-0.5,1) ;
\mydotw{(-0.5,-0.5)};
\end{tikzpicture}
\end{equation}

It was pointed out to us by an anonymous referee that the relations on the right in \ref{SL+} and \ref{SL-} follow immediately from the relations on the left by simply adding the appropriate crossing beneath the diagram, and using the fact that it is the inverse of the given crossing. Thus our presentation is redundant. Furthermore, the vertically reflected versions of the relations in \ref{SL+} and \ref{SL-} follow from the given relations and $\ref{univalent}$, which will be crucial in our arguments.

Notice that if $n$ is odd, the $\text{SL}^{\pm}_{n}$ relations agree, but if $n$ is even these relations are genuinely different. We define $\text{Web}(\text{SL}^{+}_{n})$ (respectively $\text{Web}(\text{SL}^{-}_{n})$) to be the category generated by trivalent merges, splits, and crossings as in $\text{PolyWeb}(\text{GL}_{n})$ satisfying the defining relations of $\text{PolyWeb}(\text{GL}_{n})$, and in addition isomorphisms from $n$ to the identity satisfying the above relations. Notice that the crossing in $\text{Web}(\text{SL}^{+}_{n})$ is still a braiding (in the sense of \cite[Definition 8.1.1]{MR3242743}) when $n$ is even, but when $n$ is odd it is not natural with respect to the isomorphism from $n$ to $\varnothing$.

We claim that in either of these cases, the resulting category is rigid. Indeed, using equation \ref{braiddefine} we see 

\begin{equation}
\begin{tikzpicture}[scale=0.5, baseline = -.1cm]
\draw  (-0.7,-1) node[below] {a} -- (0.7,1) node[above] {a};
\draw  ( 0.7,-1)node[below] {n} -- (-0.7,1) node[above] {n};
\end{tikzpicture}=  \ \ \ (-1)^{a}\ \ \  \begin{tikzpicture}[scale=0.5, baseline = -.1cm]
\draw (-0.8,-1) node[below] {a} -- (-0.8,1) node[above] {n};
\draw (0.8,-1) node[below] {n} -- (0.8,1) node[above] {a};
\draw (0.8,-0.2)--(-0.8,0.2);
\node(none) at (0,0.5) {\small n-a};
\end{tikzpicture}
\end{equation}

\begin{equation}
\begin{tikzpicture}[scale=0.5, baseline = -.1cm]
\draw  (-0.7,-1) node[below] {n} -- (0.7,1) node[above] {n};
\draw  ( 0.7,-1)node[below] {a} -- (-0.7,1) node[above] {a};
\end{tikzpicture}=  \ \ \ (-1)^{a}\ \ \  \begin{tikzpicture}[scale=0.5, baseline = -.1cm]
\draw (-0.8,-1) node[below] {n} -- (-0.8,1) node[above] {a};
\draw (0.8,-1) node[below] {a} -- (0.8,1) node[above] {n};
\draw (0.8,0.2)--(-0.8,-0.2);
\node(none) at (0,0.5) {\small n-a};
\end{tikzpicture}
\end{equation}

Then using the isomorphism $n$ to $\varnothing$ and assuming relation $\ref{univalent}$, its easy to see the $\text{SL}^{\pm}_{n}$ relations can be restated as follows:

\begin{equation}\label{SL+rigid}
\text{SL}^{+}_{n}\ \text{relations}:\ \ \ \ \ \ \ \ \ \ \begin{tikzpicture}[scale=0.6, baseline = -.1cm]
\draw  (0,-1) node[below] {a} -- (0,1);
\end{tikzpicture}=\ \ \ (-1)^{a(n+1)}\ \ \  \begin{tikzpicture}[scale=0.6, baseline = -.1cm]
\draw (-0.7,-1) node[below] {a} -- (-0.7,1) ;
\draw (0.7,-1) -- (0.7,1) node[above] {a};
\draw (0.7,-0.2)--(-0.7,0.2);
\node(none) at (0,0.5) {\small n-a};
\mydotw{(0.7,-1)};
\mydotw{(-0.7,1)};
\end{tikzpicture}=\ \ \ (-1)^{a(n+1)}\ \ \  
\begin{tikzpicture}[scale=0.6, baseline = -.1cm]
\draw (-0.7,-1)  -- (-0.7,1) node[above] {a};
\draw (0.7,-1) node[below] {a} -- (0.7,1);
\draw (0.7,0.2)--(-0.7,-0.2);
\node(none) at (0,0.5) {\small n-a};
\mydotw{(-0.7,-1)};
\mydotw{(0.7,1)};
\end{tikzpicture}
\end{equation}

\begin{equation}\label{SL-rigid}
\text{SL}^{-}_{n}\ \text{relations}:\ \ \ \ \ \ \ \ \ \ \begin{tikzpicture}[scale=0.6, baseline = -.1cm]
\draw  (0,-1) node[below] {a} -- (0,1);
\end{tikzpicture}=\ \ \ \ \ \  \begin{tikzpicture}[scale=0.6, baseline = -.1cm]
\draw (-0.7,-1) node[below] {a} -- (-0.7,1) ;
\draw (0.7,-1) -- (0.7,1) node[above] {a};
\draw (0.7,-0.2)--(-0.7,0.2);
\node(none) at (0,0.5) {\small n-a};
\mydotw{(0.7,-1)};
\mydotw{(-0.7,1)};
\end{tikzpicture}=\ \ \ \ \ \  
\begin{tikzpicture}[scale=0.6, baseline = -.1cm]
\draw (-0.7,-1)  -- (-0.7,1) node[above] {a};
\draw (0.7,-1) node[below] {a} -- (0.7,1);
\draw (0.7,0.2)--(-0.7,-0.2);
\node(none) at (0,0.5) {\small n-a};
\mydotw{(-0.7,-1)};
\mydotw{(0.7,1)};
\end{tikzpicture}
\end{equation}

In either case, the category $\text{Web}(\text{SL}^{\pm}_{n})$ is rigid, where the two-sided dual of the object $a$ is $n-a$. The only difference between these categories is whether or not the merges and splits can be used directly to define duality maps (otherwise, they must be appropriately normalized by a sign). 

\begin{rem}\label{tilting} If the field $\mathbbm{k}$ is algebraically closed, then the idempotent completion of $\text{Web}(\text{SL}^{+}_{n})$ is equivalent to the category of tilting modules for $\text{SL}_{n}$. To see this, recall the category of tilting modules $\text{Tilt}(\text{SL}_{n})$ is equivalent to the idempotent completion of the full subcategory of $\text{Rep}(\text{SL}_{n})$ generated by tensor products of exterior powers of the defining representation \cite{MR1200163}. The fiber functor $F$ from Remark \ref{standardfunctorpolyweb} extends to a functor $\widehat{F}$ on $\text{Web}(\text{SL}^{+}_{n})$ by picking an arbitrary isomorphism from the determinant representation (i.e. $\bigwedge^{n} V$) to the trivial representation, and assigning the new generators in $\text{Web}(\text{SL}^{+}_{n})$ to this morphism and its inverse. Since the $\text{SL}^{+}_{n}$ relation with the crossing will be satisfied, it is clear that this defines a monoidal functor from $\text{Web}(\text{SL}^{+}_{n})$ to $\text{Tilt}(\text{SL}_{n})$. It remains to show this functor is fully faithful, which follows in a straightforward way from \cite[Remark 4.15]{MR4135417}. This was pointed out to us by Victor Ostrik. We will provide details for the convenience of the reader.

Let $V$ denote the defining $n$ dimensional representation of both $\text{GL}_{n}$ and $\text{SL}_{n}$. For an object $k=(k_{1},\ \dots,\ k_{s})$ in $\text{Web}(\text{SL}^{+}_{n})$, we denote by $V_{k}=\otimes^{s}_{i=1} \bigwedge^{k_{i}} V$. For two sequences $k=(k_{1},\ \dots,\ k_{s}), m=(m_{1},\ \dots,\ m_{t})$, we claim that the functor $\widehat{F}$ induces an isomorphism $$\text{Web}(\text{SL}^{+}_{n})(k,l)\rightarrow \text{Rep}(\text{SL}_{n})(V_{k}, V_{m})\footnote{We adopt the convention that $\mathcal{C}(a,b)$ denotes the space of morphisms with source $a$ and target $b$ in a category $\mathcal{C}$.}.$$ 

Note that $\text{Rep}(\text{SL}_{n})(V_{k}, V_{l})\ne 0$ implies $\sum^{s}_{i=1} k_{i}\equiv \sum^{t}_{j=1} m_{j}\ \text{mod}\ n$. We consider the case $\sum^{s}_{i=1} k_{i}\le\sum^{t}_{j=1} m_{j}$ (the other case is completely analogous). Suppose $\sum^{s}_{i=1} k_{i} +l\cdot n= \sum^{t}_{j=1} m_{j}$ for some $l\ge 0$. Set $\overline{k}:=(k_{1}, \cdots, k_{s}, n, \cdots, n)=k\otimes n^{l}$, where we adjoin $l$ entries to the right of $k$ labeled by $n$. Then the have the following commuting diagram

 $$\begin{tikzcd}
    \text{Rep}(\text{GL}_{n})(V_{\overline{k}}, V_{m}) \arrow[swap]{d}{\text{res}} & \text{PolyWeb}(\text{GL}_{n})(\overline{k}, m)\arrow[swap]{d}{I} \arrow[swap]{l}{F} \\
    \text{Rep}(\text{SL}_{n})(V_{\overline{k}}, V_{m}) \arrow{d} & \text{Web}(\text{SL}^{+}_{n})(\overline{k}, m)\arrow{d} \arrow[swap]{l}{\widehat{F}}\\
    \text{Rep}(\text{SL}_{n})(V_{k}, V_{m}) & \text{Web}(\text{SL}^{+}_{n})(k, m)\arrow[swap]{l}{\widehat{F}}
  \end{tikzcd}
$$

\noindent where top right vertical arrow $I$ is the defining functor from polynomial $\text{GL}_{n}$ webs to $\text{SL}_{n}$ webs and the top left vertical arrow is the restriction functor. The lower left vertical arrow arises from applying our choice of isomorphism $\bigwedge^{n} V$ to the trivial representation to the $l$ extraneous factors of $n$ in $\overline{k}=k\otimes n^{l}$ (used in defining our extension of $F$ to $\text{Web}(\text{SL}^{+}_{n}))$, and the lower right vertical arrow arises by applying the univalent vertex to additional factors of $n$ in $\text{Web}(\text{SL}_{n})$. The commutativity of this diagram follows directly from the definition of $\hat{F}$.

The top arrow is an isomorphism by \cite[Remark 4.15]{MR4135417}. Recall that if $\sum^{s}_{i=1} k_{i}=\sum^{t}_{j=1} m_{j}$, then restriction $\text{res}: \text{Rep}(\text{GL}_{n})(V_{k}, V_{m})\rightarrow  \text{Rep}(\text{SL}_{n})(V_{k}, V_{m})$ is an isomorphism (this can be easily deduced by writing an arbitrary element of $\text{GL}_{n}$ as a scalar times an element of $\text{SL}_{n}$). Thus the upper left vertical arrow is an isomorphism. The lower vertical arrows are obviously isomorphisms. Hence it suffices to show the upper right arrow is surjective, which will imply to the two remaining horizontal arrows associated to $F$ are isomorphisms. 

Take a diagram $D\in \text{Web}(\text{SL}^{+}_{n})(\overline{k}, m)$, and for each univalent vertex that appears, pull it to the far right of the diagram adding an appropriate sign and a crossing when passing over strings according to the defining $\text{SL}^{+}_{n}$ relation \ref{SL+}. Attach these $n$-strands to either the top or bottom boundary of the diagram (depending on which way the univalent vertex was pointing) so as to not create any critical points in the strings. This results in a scalar multiple of a diagram $D^{\prime}\in \text{PolyWeb}(\text{GL}_{n})(k^{\prime}, m^{\prime})$ where $k^{\prime}=(k_{1}, \cdots, k_{s}, n, \cdots n)=\overline{k}\otimes n^{\otimes r}$ and $(m_{1}, \cdots, m_{t}, n, \cdots n)=m\otimes n^{\otimes r}$ with $r$ factors of $n$ adjoined to $m$, and $r+l$ adjoined to $k$.

Notice that the map $\text{PolyWeb}(\text{GL}_{n})(\overline{k}, m)\rightarrow \text{PolyWeb}(\text{GL}_{n})(\overline{k}\otimes n, m\otimes n)$ obtained by adding a vertical $n$ strand to the right of a diagram is an isomorphism, since the Web category embeds as a full subcategory of $\text{Rep}(\text{GL}_{n})$ (\cite[Remark 4.15]{MR4135417}) with $n$ corresponding to the determinant representation, which is invertible. Therefore, there exists an element $D^{\prime \prime}\in \text{PolyWeb}(\text{GL}_{n})(\overline{k}, m) $ such that $D^{\prime \prime}\otimes 1_{n^{\otimes r}}=D^{\prime}$. Placing univalent vertices on the $r$ right most $n$-strands shows that the image of $D^{\prime \prime}$ in $\text{Web}(\text{SL}^{+}_{n})(\overline{k}, m)$ is precisely $D$. This gives surjectivity of the upper right arrow as desired.

\end{rem}

\begin{rem}\label{spidercomparison} The categories $\text{Web}(\text{SL}^{\pm}_{n})$ bear some similarities to the quantum $\text{SL}_{n}$ spiders of \cite{MR3263166} (extending Kuperberg's $\text{SL}_{3}$ spider \cite{MR1403861}) at $q=\pm 1$, though the exact relationship is unclear. As we've seen above $\text{Web}(\text{SL}^{+}_{n})$ can be naturally interpreted in the context of representation theory, but a natural interpretation of $\text{Web}(\text{SL}^{-}_{n})$ is less obvious. While we do not provide such an interpretation here, we point out that for $n=2$ these categories are in fact already well known.

We briefly recall the Temperley-Lieb-Jones category with loop parameter $\delta\in \mathbbm{k}$, denoted $\text{TLJ}(\delta)$ (See \cite{MR1280463,MR1292673} or \cite{1502.06845} for a more recent treatment). $\text{TLJ}(\delta)$ is the $\mathbbm{k}$-linear diagrammatic strict monoidal categories generated by cup and cap morphisms $\begin{tikzpicture}[scale=0.6, baseline = -.1cm]
    \draw
    (0.5,0.25)
    arc(360:180:0.5);
  \end{tikzpicture}\ \text{and}\ \begin{tikzpicture}[scale=0.6, baseline = -.1cm]
    \draw
    (-0.5,-0.25)
    arc(0:180:0.5);
  \end{tikzpicture}$ satisfying the relations

 $$\begin{tikzpicture}[scale=0.6, baseline = -.1cm]
    \draw
    (-0.5,1)--(-0.5,0)
    arc(180:360:0.25)
    arc(180:0:0.25)
    (0.5,0)--(0.5,-1);
  \end{tikzpicture}\ =\ \begin{tikzpicture}[scale=0.6, baseline = -.1cm]
    \draw
    (0,-1)--(0,1);
  \end{tikzpicture}\ = \ \begin{tikzpicture}[scale=0.6, baseline = -.1cm]
    \draw
    (0.5,1)--(0.5,0)
    arc(360:180:0.25)
    arc(0:180:0.25)
    (-0.5,0)--(-0.5,-1);
  \end{tikzpicture}\ \ \ \ \ \ \ \  \text{and}\ \ \ \ \ \ \ \begin{tikzpicture}[scale=0.6, baseline = -.1cm]
    \draw
    (.25,0) arc(0:360:0.5);
  \end{tikzpicture}=\delta$$

When $\delta=2\in \mathbbm{k}$, it is straightforward to check that the assignment which sends the single string to $1$ and 

 $$\begin{tikzpicture}[scale=0.6, baseline = -.1cm]
    \draw
    (0.5,0.5)
    arc(360:180:0.5);
  \end{tikzpicture}\ \mapsto \begin{tikzpicture}[scale=0.6, baseline = -.1cm]
\draw (0,-0.5)  -- (0,0) ;
\draw (0,0) -- (-0.5,0.5) node[above] {1};
\draw (0,0)--(0.5,0.5) node[above] {1};
\mydotw{(0,-0.5)};
\end{tikzpicture}\ \ \ \text{and}\ \ \  \begin{tikzpicture}[scale=0.6, baseline = -.1cm]
    \draw
    (-0.5,-0.25)
    arc(0:180:0.5);
  \end{tikzpicture}\ \mapsto \begin{tikzpicture}[scale=0.6, baseline = -.1cm]
\draw (-0.5,-0.5) node[below] {1}  -- (0,0) ;
\draw (0.5,-0.5) node[below] {1} -- (0,0);
\draw (0,0)--(0,0.5) ;
\mydotw{(0,0.5)};
\end{tikzpicture}$$ 

\noindent extends to an equivalence $\text{TLJ}(2)\cong \text{Web}(\text{SL}^{-}_{2})$. Similarly for loop parameter $\delta=-2\in \mathbbm{k}$, the assignment

$$\begin{tikzpicture}[scale=0.6, baseline = -.1cm]
    \draw
    (0.5,0.5)
    arc(360:180:0.5);
  \end{tikzpicture}\ \mapsto \begin{tikzpicture}[scale=0.6, baseline = -.1cm]
\draw (0,-0.5)  -- (0,0) ;
\draw (0,0) -- (-0.5,0.5) node[above] {1};
\draw (0,0)--(0.5,0.5) node[above] {1};
\mydotw{(0,-0.5)};
\end{tikzpicture}\ \ \ \text{and}\ \ \  \begin{tikzpicture}[scale=0.6, baseline = -.1cm]
    \draw
    (-0.5,-0.25)
    arc(0:180:0.5);
  \end{tikzpicture}\ \mapsto - \begin{tikzpicture}[scale=0.6, baseline = -.1cm]
\draw (-0.5,-0.5) node[below] {1}  -- (0,0) ;
\draw (0.5,-0.5) node[below] {1} -- (0,0);
\draw (0,0)--(0,0.5) ;
\mydotw{(0,0.5)};
\end{tikzpicture}$$ 

\noindent extends to a monoidal equivalence $\text{TLJ}(-2)\cong \text{Web}(\text{SL}^{+}_{2})$.

\end{rem}

\subsection{$\text{Vec}(\Gamma)$}

Let $\Gamma$ be a group and fix a field $\mathbbm{k}$. Then the rigid monoidal category $\text{Vec}(\Gamma)$ is the category of finite dimensional \textit{$\Gamma$-graded vector spaces} over $\mathbbm{k}$. Objects are finite dimensional vector spaces with a $\Gamma$-grading and morphisms are linear maps that respect the grading.

For two objects $V=\bigoplus_{g\in \Gamma} V_{g}$, $W=\bigoplus_{g\in \Gamma} W_{g}$, the monoidal product has $h$ graded component $(V\otimes W)_{h}:=\bigoplus_{g\in \Gamma} V_{g}\otimes W_{g^{-1}h}$. This direct sum is finite since only finitely many components of each vector space are non-zero. The associator is inherited from $\text{Vec}$. This category is rigid, with $(V_{g})^{*}\cong (V^{*})_{g^{-1}}$ with evaluation and coevalution inherited from $\text{Vec}$. The forgetful functor $\text{Forget}:\text{Vec}(\Gamma)\rightarrow \text{Vec}$ simply forgets the grading of the underlying vector spaces.

Simple objects in $\text{Vec}_{\mathbbm{k}}(\Gamma)$ are isomorphic to a copy of $\mathbbm{k}$ graded by elements $g\in \Gamma$, denoted $\mathbbm{k}_{g}$. If $gh=k$, then the definition of the monoidal product gives us canonical isomorphisms $\lambda^{k}_{g,h}:\mathbbm{k}_{g}\otimes \mathbbm{k}_{h}\rightarrow \mathbbm{k}_{k}$, defined by the linear extension of $1_{g}\otimes 1_{h}\mapsto 1_{gh} $. We draw these isomorphisms and their inverses as trivalent vertices labeled by group elements:

\begin{equation}\label{groupvertices}
\lambda^{k}_{g,h}=:\begin{tikzpicture}[scale=0.7, baseline = -.1cm]
\draw (90:1) node[above] {k}  -- (0,0);
\draw (220:1) node[below] {g} -- (0,0);
\draw (320:1) node[below] {h}  -- (0,0);
\end{tikzpicture}\ \ \ \  \text{and}\ \ \ \ \ 
\text{y}^{g,h}_{k}=:\begin{tikzpicture}[scale=0.7, baseline = -.1cm]
\draw (40:1) node[above] {g} -- (0,0);
\draw (140:1) node[above] {h} -- (0,0);
\draw (270:1)  node[below] {k} -- (0,0);
\end{tikzpicture}
\end{equation}

That these are mutually inverse means $\lambda^{k}_{g,h}\circ \text{y}^{g,h}_{k}=1_{k}$ and $\text{y}^{g,h}_{k}\circ \lambda^{k}_{g,h}=1_{g}\otimes 1_{h}$, graphically depicted by the relations

\begin{equation}
\begin{tikzpicture}[scale=0.7, baseline = -.1cm]
\draw (-0.5,-1) node[below] {g}  -- (0,-0.3);
\draw (0.5,-1) node[below] {h} -- (0,-0.3);
\draw (0,-0.3)--(0,0.3);
\draw (0,0.3)--(-0.5,1);
\draw (0,0.3)--(0.5,1);
\end{tikzpicture}\ \ \ \ = \ \ \ \ 
\begin{tikzpicture}[scale=0.7, baseline = -.1cm]
\draw (-0.5,-1) node[below] {g}  -- (-0.5,1);
\draw (0.5,-1) node[below] {h} -- (0.5,1);
\end{tikzpicture},\ \ \ \ \ \ \ \ \ \ 
\begin{tikzpicture}[scale=0.7, baseline = -.1cm]
\draw (0,-1)node[below]{a+b} -- (0,-0.5) ;
\draw (270:0.5) arc (270:90:0.5);
\draw (-90:0.5) arc (-90:90:0.5);
\draw (0,0.5)--(0,1) node[above]{gh};
\node(none) at (1,0) {g};
\node(none) at (-1,0) {h};
\end{tikzpicture}\ \ \ \ = \ \ \ 
\begin{tikzpicture}[scale=0.7, baseline = -.1cm]
\draw (0,-1)--(0,1);
\end{tikzpicture}
\end{equation}

These isomorphisms satisfy the associativity conditions given by the following diagrams (and ignoring vector space associators)

\begin{equation}\label{groupassociativity}
\begin{tikzpicture}[scale=0.7, baseline = 0 cm]
\draw (-1,1) node[above]{g} -- (-0.5,0.5);
\draw (-0.5,0.5) --  (0,0);
\draw (0,1)node[above]{h}--(-0.5,0.5);
\draw (1,1)node[above]{k}--(0,0);
\draw (0,0)--(0,-0.5)node[below]{ghk};
\node(none) at (-0.6,0) {gh};
\end{tikzpicture}\ \ \ =\ \ \ 
\begin{tikzpicture}[scale=0.7, baseline = 0cm]
\draw (-1,1) node[above]{g} -- (0,0) ;
\draw (0.5,0.5) -- (0,0);
\draw (0,1)node[above]{h}--(0.5,0.5);
\draw (1,1)node[above]{k}--(0.5,0.5);
\draw (0,0)--(0,-0.5)node[below]{ghk};
\node(none) at (0.6,0) {hk};
\end{tikzpicture}\ \ \ \  ,\ \ \ \ \ \ \ \ \ 
\begin{tikzpicture}[scale=0.7, baseline = -.3 cm]
\draw (-1,-1) node[below]{g} -- (-0.5,-0.5);
\draw (-0.5,-0.5) --  (0,0);
\draw (0,-1)node[below]{h}--(-0.5,-0.5);
\draw (1,-1)node[below]{k}--(0,0);
\draw (0,0)--(0,0.5)node[above]{ghk};
\node(none) at (-0.6,0) {gh};
\end{tikzpicture}\ \ \ =\ \ \ 
\begin{tikzpicture}[scale=0.7, baseline = -.3cm]
\draw (-1,-1) node[below]{g} -- (0,0) ;
\draw (0.5,-0.5) -- (0,0);
\draw (0,-1)node[below]{h}--(0.5,-0.5);
\draw (1,-1)node[below]{k}--(0.5,-0.5);
\draw (0,0)--(0,0.5)node[above]{ghk};
\node(none) at (0.6,0) {hk};
\end{tikzpicture}
\end{equation}

Now, let $\mathcal{T}$ be a triangle presentation, and $\Gamma=\Gamma_{\mathcal{T}}$. 

\begin{defn}\label{trianglemorphism} An isomorphism $\text{y}^{u,v}_{w}$ or $\lambda^{w}_{u,v}$ in $\text{Vec}(\Gamma)$ is a \textit{triangle morphism} if $u,v,w\in \Pi$ and $(u,v,\sigma(w))\in \mathcal{T}$. 
\end{defn}


\begin{rem}\label{interpretweirdone} Using the graphical calculus for $\text{Vec}(\Gamma)$, we can now give a natural interpretation to the condition \ref{theweirdone}, Definition \ref{tridefn}. Let $u,v,w,p,q\in \Pi$ be generators of $\Gamma$, and suppose $(u,v,\sigma(p)), (p,w,\sigma(q))\in \mathcal{T}$. Then $uv=p$ and $pw=q$, so we have an isomorphism in $\text{Vec}(\Gamma)$

\begin{equation}
(\text{y}^{u,v}_{p}\otimes 1_{w})\circ \text{y}^{p,w}_{q}: \mathbbm{k}_{q}\rightarrow \mathbbm{k}_{u}\otimes \mathbbm{k}_{v}\otimes \mathbbm{k}_{w}
\end{equation}

The associativity condition in $\text{Vec}(\Gamma)$ above tells us there is an element of the group $g=uv\in \Gamma$ so that 

\begin{equation}\label{trianglegroupassoc}
(\text{y}^{u,v}_{p}\otimes 1_{w})\circ \text{y}^{p,w}_{q}=(1_{u}\otimes \text{y}^{v,w}_{g})\circ \text{y}^{u,g}_{q}
\end{equation}

\noindent What the condition \ref{theweirdone}, Definition \ref{tridefn} (or rather, its equivalent conditions from Proposition \ref{betterweirdone}) tells us is that if $\text{dim}(u)+\text{dim}(v)<n$ and $\text{dim}(p)+\text{dim}(w)<n$, we can choose $g\in \Pi$ so that $\text{y}^{v,w}_{g}$ and $\text{y}^{u,g}_{q}$ are \textit{triangle morphisms}, i.e. $(u, g, \sigma(q)), (v,w,\sigma(g))\in \mathcal{T}$.
\end{rem}

\section{The functor}\label{functor}

Let $\mathcal{T}$ be a triangle presentation of type $\tilde{\text{A}}_{n-1}$ on a projective geometry $\Pi$ of order $q$. We introduce the notation

$$\mathcal{T}_{a,b}:=\{(u,v,w)\in \mathcal{T}\ : u\in \Pi_{a},\ v\in \Pi_{b}\}.$$

If $(u,v,w)\in \mathcal{T}_{a,b}$ then $w\in \Pi_{n-a-b}$ if $a+b<n$ and $w\in \Pi_{2n-a-b}$ if $a+b>n$ (note $a+b\ne n$). Now, define the objects in $\text{Vec}(\Gamma)$

\begin{align*}
V_{a}&=\bigoplus_{u\in \Pi_{a}} \mathbbm{k}_{u}\ \ \text{for}\ 0<a<n\\
V_{n}&=\mathbbm{k}_{1}
\end{align*}

\medskip

We will define a monoidal functor from $\text{Web}(\text{SL}^{-}_{n})\rightarrow \text{Vec}(\Gamma)$ by assigning the generating objects $a\mapsto V_{a}$ and the monoidal unit $\varnothing \mapsto \mathbbm{k}_{1}$. The generating morphisms are mapped to linear maps between tensor powers of the $V_{a}$ as follows:

\begin{equation}\label{trianglemerge}
\text{For}\ a+b<n,\ \ \ \ \ \ \ \begin{tikzpicture}[scale=0.5, baseline = -.1cm]
\draw (90:1) node[above] {a+b}  -- (0,0);
\draw (220:1) node[below] {a} -- (0,0);
\draw(320:1) node[below] {b}  -- (0,0);
\end{tikzpicture}\mapsto \ \ \ \ \ \ \ \ \  \bigoplus_{(u,v,\sigma(w))\in\mathcal{T}_{a,b}} \lambda^{w}_{u,v} \ \ \ \ :\ \ \ \ V_{a}\otimes V_{b}\rightarrow V_{a+b}
\end{equation}

\bigskip

\begin{equation}
\begin{tikzpicture}[scale=0.5, baseline = -.1cm]
\draw (90:1) node[above] {n}  -- (0,0);
\draw (220:1) node[below] {a} -- (0,0);
\draw(320:1) node[below] {n-a}  -- (0,0);
\end{tikzpicture}\mapsto \ \ \ \ \ \ \ \ \ \bigoplus_{u\in \Pi_{a}}\ \ \ \  \lambda^{1}_{u, \sigma(u)}\ \ \ \ :\ \ \ \ V_{a}\otimes V_{n-a}\rightarrow V_{n}=\mathbbm{k}_{1}
\end{equation}

\bigskip

\begin{equation}\label{trianglesplit}
\displaystyle
\text{For}\ a+b<n,\ \ \ \ \ \ \ \begin{tikzpicture}[scale=0.5, baseline = -.1cm]
 \draw (40:1) node[above] {b} -- (0,0);
\draw (140:1) node[above] {a} -- (0,0);
\draw (270:1)  node[below] {a+b} -- (0,0);
\end{tikzpicture}\mapsto \ \ \ \ \ \ \ \ \  \bigoplus_{(v,w,\sigma(u))\in\mathcal{T}_{a, b}} \text{y}^{v,w}_{u}  : V_{a+b}\rightarrow V_{a}\otimes V_{b} 
\end{equation}

\bigskip

\begin{equation}\label{trianglesplitn}
\begin{tikzpicture}[scale=0.5, baseline = -.1cm]
\draw (40:1) node[above] {n-a} -- (0,0);
\draw (140:1) node[above] {a} -- (0,0);
\draw (270:1)  node[below] {n} -- (0,0);
\end{tikzpicture}\mapsto \ \ \ \ \ \ \ \ \ \bigoplus_{u\in \Pi_{a}}\ \ \ \ \  \text{y}^{u,\sigma(u)}_{1}: V_{n}\rightarrow V_{a}\otimes V_{n-a} 
\end{equation}

\bigskip

\begin{equation}\label{triangleunivalent}
\begin{tikzpicture}[scale=0.5, baseline = -.1cm]
\draw (0,-0.5) node[below]{n}--(0,0.5);
\mydotw{(0,0.5)};
\end{tikzpicture}\mapsto \ \ \ \ \ \ \ \ \ id_{\mathbbm{k}} : V_{n}=\mathbbm{k}_{1}\rightarrow \mathbbm{k}_{1}
\end{equation}

\medskip

\begin{equation}\label{triangleunivalent2}
\begin{tikzpicture}[scale=0.5, baseline = -.1cm]
\draw (0,-0.5) --(0,0.5)node[above]{n};
\mydotw{(0,-0.5)};
\end{tikzpicture}\mapsto \ \ \ \ \ \ \ \ \ id_{\mathbbm{k}} : \mathbbm{k}_{1}\rightarrow \mathbbm{k}_{1}=V_{n}
\end{equation}

\bigskip

\noindent where $\lambda$ and $\gamma$ are triangle morphisms in $\text{Vec}(\Gamma)$, defined in \ref{trianglemorphism}. We will show these assignments extend to a functor from $\text{Web}(\text{SL}^{-}_{n})\rightarrow \text{Vec}(\Gamma)$ when the characteristic of the field satisfies the appropriate conditions. First we have the following Lemma which is the main technical part of the paper.

\begin{lem}\label{squarelemma1} If $\text{char}(\mathbbm{k})=p$ satisfies $q\equiv 1\ \text{mod}\ p$, the following special cases of the square switch relation are satisfied by the morphisms \ref{trianglemerge}-\ref{trianglesplitn} defined above in the category $\text{Vec}(\Gamma)$.
\end{lem}

\begin{equation}\label{squareswitcha1}\begin{tikzpicture}[scale=0.7, baseline = -.1cm]
\draw (-0.5,-1.2) -- (-0.5,-0.6) ;
\draw (-0.5,-0.6) -- (0.5,-0.4) ;
\draw (-0.5,-0.6) -- (-0.5,0.3);
\draw (-0.5,0.3) -- (-0.5,1.2);
\draw (0.5,0.4) -- (-0.5,0.6);
\draw (0.5,-1.2)--(0.5,-0.2);
\draw (0.5,-0.2)--(0.5,0.6);
\draw (0.5,0.6)--(0.5,1.2);
\node(none) at (-1,0) {a-1};
\node(none) at (1,0) {2};
\node(none) at (-0.9,1.2) {a};
\node(none) at (0.9,1.2) {1};
\node(none) at (-0.9,-1.2) {a};
\node(none) at (0.8,-1.2) {1};
\end{tikzpicture}\ =\ 
\begin{tikzpicture}[scale=0.7, baseline = -.1cm]
\draw (-0.5,-1.2) -- (0,-0.6) ;
\draw (0,-0.6) -- (0,0.6) ;
\draw (0,0.6) -- (-0.5,1.2);
\draw (0,0.6) -- (0.5,1.2);
\draw (0.5,-1.2)--(0,-0.6);
\node(none) at (0.7,0) {a+1};
\node(none) at (-0.9,1.2) {a};
\node(none) at (0.9,1.2) {1};
\node(none) at (-0.9,-1.2) {a};
\node(none) at (0.8,-1.2) {1};
\end{tikzpicture}\ +\ 
(a-1)\ \begin{tikzpicture}[scale=0.7, baseline = -.1cm]
\draw (-0.5,-1.2) -- (-0.5,1.2) ;
\draw (0.5,-1.2) -- (0.5, 1.2) ;
\node(none) at (-0.9,-1.2) {a};
\node(none) at (0.8,-1.2) {1};
\end{tikzpicture}
\end{equation}

\begin{equation}\label{squareswitch1a}
\begin{tikzpicture}[scale=0.7, baseline = -.1cm]
\draw (-0.5,-1.2) -- (-0.5,-0.6) ;
\draw (0.5,-0.6) -- (-0.5,-0.4) ;
\draw (-0.5,-0.6) -- (-0.5,0.3);
\draw (-0.5,0.3) -- (-0.5,1.2);
\draw (-0.5,0.4) -- (0.5,0.6);
\draw (0.5,-1.2)--(0.5,-0.2);
\draw (0.5,-0.2)--(0.5,0.6);
\draw (0.5,0.6)--(0.5,1.2);
\node(none) at (-1,0) {2};
\node(none) at (1,0) {a-1};
\node(none) at (-0.9,1.2) {1};
\node(none) at (0.9,1.2) {a};
\node(none) at (-0.9,-1.2) {1};
\node(none) at (0.8,-1.2) {a};
\end{tikzpicture}\ =
\begin{tikzpicture}[scale=0.7, baseline = -.1cm]
\draw (-0.5,-1.2) -- (0,-0.6) ;
\draw (0,-0.6) -- (0,0.6) ;
\draw (0,0.6) -- (-0.5,1.2);
\draw (0,0.6) -- (0.5,1.2);
\draw (0.5,-1.2)--(0,-0.6);
\node(none) at (0.7,0) {a+1};
\node(none) at (-0.9,1.2) {1};
\node(none) at (0.9,1.2) {a};
\node(none) at (-0.9,-1.2) {1};
\node(none) at (0.8,-1.2) {a};
\end{tikzpicture}\ +\ 
(a-1)\ \begin{tikzpicture}[scale=0.7, baseline = -.1cm]
\draw (-0.5,-1.2) -- (-0.5,1.2) ;
\draw (0.5,-1.2) -- (0.5, 1.2) ;
\node(none) at (-0.9,-1.2) {1};
\node(none) at (0.8,-1.2) {a};
\end{tikzpicture}
\end{equation}

\begin{proof}
We prove \ref{squareswitcha1}. Let $\text{LHS}$ and $\text{RHS}$ denote the linear operators in $\text{End}(V_{a}\otimes V_{1})$ defined by the left-hand and right-hand side of equation \ref{squareswitcha1} respectively. For $(z,u)\in \Pi_{a}\times \Pi_{1}$ let 

$$LHS(z\otimes u)=\sum_{(w,v)\in \Pi_{a}\times \Pi_{1}} L^{w,v}_{z,u} w\otimes v$$
$$RHS(z\otimes u)=\sum_{(w,v)\in \Pi_{a}\times \Pi_{1}} R^{w,v}_{z,u} w\otimes v$$

\noindent We will show $L^{w,v}_{z,u}=R^{w,v}_{z,u}$ in cases.

\medskip

\textbf{Corner cases}. First we consider corner cases. The case $a=1$ is trivial. Suppose $a=n$. Then since $V_{n}=\mathbbm{k}$,

$$L^{1,v}_{1,u}=\delta_{u=v}\ |\{r\in \Pi_{1}\ :\ u\sim \sigma(r)\}|=\delta_{u=v}\left[ {\begin{array}{c}
   n-1 \\
   n-2  \\
  \end{array} } \right]_{q}.$$

\noindent where the last equality uses Lemma \ref{intermediatesubspacenumber}. As $q=1$ in $\mathbbm{k}$, this reduces to $n-1$. Thus $L^{1,v}_{1,u}= \delta_{u=v} n-1$. Since the morphism defined by the first diagram on the right-hand side vanishes, we immediately see $R^{1,v}_{1,u}=\delta_{u=v}\ n-1$ as desired.

\medskip

\noindent Now assume $1<a<n$.  

\medskip

\textbf{Case 1}: $z\ne w$. In this case, the morphism described by the two parallel vertical on the right-hand side contributes $0$ to $R^{w,v}_{u,z}$. From a simple examination of the diagram on the left and the remaining diagram on the right we compute 
\small

\label{computingL}
$$L^{w,v}_{z,u}=|\ \{ (p,q,r,s)\in \Pi_{a-1}\times \Pi_{2}\times \Pi_{1}\times \Pi_{1}\ :\ \\
 (z,\sigma(r),\sigma(p)),(r,u,\sigma(q)),(p,s,\sigma(w)),(s,v,\sigma(q))\in \mathcal{T}\}|
$$

\label{computeR}
$$R^{w,v}_{z,u}=|\{p^{\prime}\in \Pi_{a+1}\ :\ (z,u,\sigma(p^{\prime})), (w,v, \sigma(p^{\prime}))\in \mathcal{T}\}|$$

\normalsize

We can interpret the elements $(p,q,r,s)$ and $p^{\prime}$ counted on the right-hand side of the above equalities as counting labellings of the diagrams below, such that each vertex corresponds to a triangle morphism:

$$
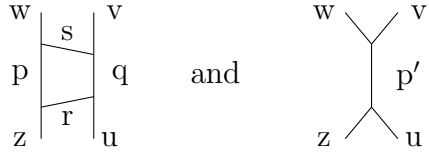
\begin{tikzpicture}[scale=0.7, baseline = -.1cm]
\draw (-0.5,-1.2) -- (-0.5,-0.6) ;
\draw (-0.5,-0.6) -- (0.5,-0.4) ;
\draw (-0.5,-0.6) -- (-0.5,0.3);
\draw (-0.5,0.3) -- (-0.5,1.2);
\draw (0.5,0.4) -- (-0.5,0.6);
\draw (0.5,-1.2)--(0.5,-0.2);
\draw (0.5,-0.2)--(0.5,0.6);
\draw (0.5,0.6)--(0.5,1.2);
\node(none) at (-0.9,0) {p};
\node(none) at (1,0) {q};
\node(none) at (-0.9,1.2) {w};
\node(none) at (0.9,1.2) {v};
\node(none) at (-0.9,-1.2) {z};
\node(none) at (0.8,-1.2) {u};
\node(none) at (0,-0.8) {r};
\node(none) at (0,0.8) {s};
\end{tikzpicture}\ \ \ \ \ \text{and}\ \ \ \ \ \ 
\begin{tikzpicture}[scale=0.7, baseline = -.1cm]
\draw (-0.5,-1.2) -- (0,-0.6) ;
\draw (0,-0.6) -- (0,0.6) ;
\draw (0,0.6) -- (-0.5,1.2);
\draw (0,0.6) -- (0.5,1.2);
\draw (0.5,-1.2)--(0,-0.6);
\node(none) at (0.7,0) {$\text{p}^{\prime}$};
\node(none) at (-0.9,1.2) {w};
\node(none) at (0.9,1.2) {v};
\node(none) at (-0.9,-1.2) {z};
\node(none) at (0.8,-1.2) {u};
\end{tikzpicture}
$$

\bigskip

We first claim that $z\ne w$ implies both $L^{w,v}_{z,u}$ and $R^{w,z}_{z,u}$ are at most $1$. For $L$, the condition \ref{uniqueness}, Definition \ref{tridefn} implies that at $q,r,s$ are determined by $p$ (if they exist). Furthermore $(z,\sigma(r),\sigma(p)),(p,s,\sigma(w))\in \mathcal{T}$ implies $z$ and $w$ are both incident with $p$ (by conditions \ref{incidence} and \ref{cyclicinvariance}, Definition \ref{tridefn}), so $p\le z\cap w$. But $z\ne w$ implies $z\cap w$ is a proper subspace of both $z$ and $w$, so $\text{dim}(z\cap w)<a$. However $\text{dim}(p)=a-1$ and so $\text{dim}(z\cap w)\ge a-1$, and thus we must have $p=z\cap w$, hence p is uniquely determined. Hence $L^{w,z}_{z,u}=1$ if $\text{dim}(z\cap w)=a-1$ and corresponding $q,r,s$ exist, and $0$ otherwise.

For $R$, note that such a $p^{\prime}$ (if it exists) is an $a+1$ dimensional subspace containing $z$ and $w$, and since $z\ne w$, we must have $p^{\prime}=z+w$ (where by the latter we mean the subspace generated by $z$ and $w$). Therefore $R^{w,v}_{z,u}$ either takes the value $0$ or $1$, as desired.

\medskip

We will now show $R^{w,v}_{z,u}\ne 0$ implies $L^{w,v}_{z,u}\ne 0$ (which by our above argument implies they must both be $1$). If $R^{w,v}_{z,u}\ne 0$ then $p^{\prime}=z+w$ has dimension $a+1$. But $a+1=\text{dim}(z+w)=\text{dim}(z)+\text{dim}(w)-\text{dim}(z\cap w)$, which implies $\text{dim}(z\cap w)=a-1$, hence $z\cap w\in \Pi_{a-1}$. Therefore there exists $r,s\in \Pi_{1}$ such that $(z\cap w, r, \sigma(z)), (z\cap w, s, \sigma(w))\in \mathcal{T}$.

Assume $a<n-1$. Then we can apply Proposition \ref{betterweirdone} to the pairs $$(z\cap w, r, \sigma(z)),\ (z,u,\sigma(z+w))\in \mathcal{T}\ \text{and}\ (z\cap w, s, \sigma(w)),\ (w,v,\sigma(z+w))\in \mathcal{T}$$  to obtain unique $q, x$ such that  $$(r,u,\sigma(q)),\ (z\cap w, q, \sigma(z+w))\in \mathcal{T}\ \text{and}\ (s,v,\sigma(x)),\ (z\cap w, x, \sigma(z+w))\in \mathcal{T}$$  respectively. But by \ref{uniqueness} in the definition of triangle presentation, $q=x$ so we have the triples $(z,\sigma(r),\sigma(z\cap w))$,$(r,u,\sigma(q))$,$(z\cap w,s,\sigma(w))$,$(s,v,\sigma(q))$ are elements of $\mathcal{T}$ hence $L^{w,v}_{z,u}\ne 0$.

If $a=n-1$, then the assumption $R^{w,v}_{z,u}\ne 0$ implies $u=\sigma(z)$ and $v=\sigma(w)$. Since $z\ne w$, $z+w$ is the entire space so we must have $\text{dim}(z\cap w)=n-2$. Then setting $q=\sigma(z\cap w), p=z\cap w$ gives $L^{w,v}_{z,u}\ne 0$ in this case.

\medskip

Now, we'll show $L^{w,v}_{z,u}\ne 0$ implies $R^{w,v}_{z,u}\ne 0$. Again assume $1<a<n-1$. Then we must have $z\cap w\in \Pi_{a-1}$, so there exists $r,s\in \Pi_{1}, q\in \Pi_{2}$ so that the triples $(z,\sigma(r),\sigma(z\cap w))$, $(r,u,\sigma(q))$, $(z\cap w,s,\sigma(w))$, and $(s,v,\sigma(q))$ are elements of $\mathcal{T}$. Our goal is to show $z+w$ lies in $\Pi_{a+1}$ and that the triples $(z,u,\sigma(z+w))$, $(w,v,\sigma(z+w))$ are elements of $\mathcal{T}$.

Notice $z\cap w\in \Pi_{a-1}$ implies $\text{dim}(z+w)=\text{dim}(z)+\text{dim}(w)-\text{dim}(z\cap w)=2a-a+1=a+1$, which is our first criteria. Let $k, l\in \Pi_{1}$ be the unique elements so that $(z,k,\sigma(z+w)),(w,l,\sigma(z+w))\in \mathcal{T}$ (which exist by condition \ref{incidence}, Definition \ref{tridefn}). Then applying Proposition \ref{betterweirdone} to the pairs $$(z\cap w,r,\sigma(z)),(z,k,\sigma(z+w))\in \mathcal{T}\ \text{and}\ (z\cap w,s,\sigma(w)),(w,l,\sigma(z+w))\in \mathcal{T}$$ there are unique $t,h\in \Pi_{2}$ such that $$(r,k,\sigma(t)), (z\cap w, t, \sigma(z+w))\in \mathcal{T}\ \text{and}\ (s,l, \sigma(h)),(z\cap w, h, \sigma(z+w))\in \mathcal{T}$$

But condition \ref{uniqueness}, Definition \ref{tridefn} gives $t=h$ and hence $t$ is incident with both $r$ and $s$. The assumption $z\ne w$ implies $r\ne s$. Since $r,s$ are 1-dimensional and not equal, $t=s+r$. $q\in \Pi_{2}$ is also incident with both $s$ and $r$ and so $q=s+r=t$. Therefore $k=u$ and $l=v$, and $(z,u,\sigma(z+w)),(w,v,\sigma(z+w))\in \mathcal{T}$.

If $a=n-1$, we repeat the same argument, except we directly see $k=\sigma(z), l=\sigma(w)$, and we have directly $t=h=\sigma(z\cap w)=q$ without needing to apply Proposition \ref{betterweirdone}. This concludes Case 1.

\medskip

\textbf{Case 2}: $z=w$. In this case, each of the conditions $L^{z,v}_{z,u}\ne 0$ and $R^{z,v}_{z,u}\ne 0$ independently imply $v=u$. Thus if $v\ne u$, both coefficients are $0$. We assume $v=u$ and calculate 

\begin{align*}
R^{z,u}_{z,u}&=\begin{cases}
a-1 & \sigma(z)\nsim u \\
a & \sigma(z)\sim u
\end{cases}
\end{align*}

and

\begin{align*}
L^{z,u}_{z,u}&=|\{r\in \Pi_{1}\ :\ (z,\sigma(r),\sigma(p)),(r,u,\sigma(q))\in \mathcal{T}\ \text{for some}\ p,q\in \Pi\}|\\
&=|\{r\in \Pi_{1}\ :\ \sigma(r)\sim\sigma(z)\ \text{and}\ \sigma(r)\sim u   \}|
\end{align*}

If $\sigma(z)\nsim u$, then since $\text{dim}(u)=1$, this implies $\sigma(z)\cap u=\{0\}$, and thus $\sigma(z)+u$ is an $n-a+1$ dimensional subspace, and an $n-1$ dimensional subspace $\sigma(r)$ is incident with both $\sigma(z)$ and $u$ if and only if it contains $\sigma(z)+u$. Thus it suffices to count the number of $n-1$ dimensional subspaces containing a given $n-a+1$ dimensional subspaces, which by Lemma \ref{intermediatesubspacenumber} is $\left[ {\begin{array}{c}
   a-1 \\
   a-2  \\
  \end{array} } \right]_{q}$, which in $\mathbbm{k}$ is equal to $a-1$ as desired.

 If $\sigma(z)\sim u$, then $u\le \sigma(z)$, and thus an $n-1$ dimensional subspaces incident to both $\sigma(z)$ and $u$ if and only if it contains the $n-a$ dimensional subspace $\sigma(z)$. Using Lemma \ref{intermediatesubspacenumber}, this is $\left[ {\begin{array}{c}
   a \\
   a-1  \\
  \end{array} } \right]_{q}$ which in $\mathbbm{k}$ reduces to $a$. This concludes the proof of \ref{squareswitcha1}. 

  \medskip

  The argument for \ref{squareswitch1a} has the same basic structure as the argument we've just given for \ref{squareswitcha1}, so we leave it to the reader.

\bigskip

\bigskip

\end{proof}


The following is our main result.

\begin{thm}\label{TriangleSLnThm}
Suppose $\mathbbm{k}$ has characteristic $p\ge n-1$ and $q\equiv 1$ mod $p$. Then linear maps defined in \ref{trianglemerge}-\ref{triangleunivalent2} satisfy the relations \ref{associativity}, \ref{burstingbigons}, \ref{squareswitch1}, \ref{squareswitch2}, \ref{SL-rigid} hence define a monoidal functor $\text{Web}(\text{SL}^{-}_{n})\rightarrow \text{Vec}(\Gamma)$, which composing with the forgetful functor to $\text{Vec}$ yields a fiber functor.
\end{thm}

\begin{proof}
The coassociativity and associativity relations from equation \ref{associativity} follow from Proposition \ref{betterweirdone} via Remark \ref{interpretweirdone}. The rigid version of the $\text{SL}^{-}_{n}$ relation from equation \ref{SL-rigid} follows directly from the fact that $\sigma$ is an involution. To verify the bigon bursting relation from equation \ref{burstingbigons}, note by the uniqueness condition \ref{uniqueness}, Definition \ref{tridefn}, we have $\begin{tikzpicture}[scale=0.5, baseline = -.1cm]
\draw (0,-1)node[below]{a+b} -- (0,-0.5) ;
\draw (270:0.5) arc (270:90:0.5);
\draw (-90:0.5) arc (-90:90:0.5);
\draw (0,0.5)--(0,1) node[above]{a+b};
\node(none) at (1,0) {a};
\node(none) at (-1,0) {b};
\end{tikzpicture}$ evaluated at $u\in \Pi_{a+b}$ is simply 
$$|\{(v,w,\sigma(u))\in \mathcal{T}\ : (v,w)\in \Pi_{a}\times \Pi_{b}\}| u=|\{w\in \Pi_{b}\ : \sigma(w)\sim \sigma(u)\}|\  u. $$

\noindent But $\sigma(u)$ has dimension $n-a-b$ and $\sigma(w)$ has dimension $n-b$, so we need the number of $n-b$ dimensional subspaces of an $n$ dimensional space containing a fixed $n-a-b$ dimensional space $\sigma(u)$. By Lemma \ref{intermediatesubspacenumber} this is $\left[ {\begin{array}{c}
   a+b \\
   a  \\
\end{array} } \right]_{q}$. But in $\mathbbm{k}$ this reduces to the usual binomial coefficient, which gives \ref{burstingbigons}. It therefore remains to prove the square switch relations \ref{squareswitch1} and \ref{squareswitch2}. The special cases with $1\le c, d\le n-2$ follow from Lemma \ref{squarelemma1} above, combined with Proposition \ref{easytohardsquareswitch}. So it remains to verify the cases with $c$ or $d$ equal to $n-1$.

We consider the case of \ref{squareswitch1} with $d=n-1$, the other 3 cases being similar. In this case, we must have $a=n$, $b=1$, and thus the relation reduces to 

\small

\begin{equation}
 \begin{tikzpicture}[scale=0.7, baseline = -.1cm]
\draw (-0.5,-1) node[below] {n} -- (-0.5,-0.5) ;
\draw (-0.5,-0.5) -- (0.5,-0.2) ;
\draw (-0.5,-0.5) -- (-0.5,0.5);
\draw (-0.5,0.5) -- (-0.5,1) node[above] {c+1};
\draw (0.5,0.2) -- (-0.5,0.5);
\draw (0.5,-1) node[below] {1} -- (0.5,-0.2);
\draw (0.5,-0.2)--(0.5,0.5);
\draw (0.5,0.5)--(0.5,1) node[above] {n-c} ;
\node(none) at (0,0.7) {c};
\node(none) at (0,-0.7) {n-1};
\end{tikzpicture}= \ \begin{tikzpicture}[scale=0.7, baseline = -.1cm]
\draw (-0.6,-1) node[below]{n}--(-0.6,1) node[above]{c+1};
\draw (0.6,-1) node[below]{1}--(0.6,1) node[above]{n-c};
\draw (-0.6,0.2)--(0.6,0.5);
\node(none) at (0,0) {n-1-c};
\end{tikzpicture}
\end{equation}

\normalsize

Using the $\text{SL}^{-}_{n}$ relations \ref{univalent} and \ref{SL-rigid} this equation becomes

\small

\begin{equation}
 \begin{tikzpicture}[scale=0.7, baseline = -.1cm]
\draw (-0.5,-1) node[below]{1}--(-0.5,1) node[above]{c+1};
\draw (-0.5,0.6) -- (0.5,0.2);
\draw (0.5,0) -- (0.5,1) node[above]{n-c};
\mydotw{(0.5,0)};
\end{tikzpicture}= \ \begin{tikzpicture}[scale=0.7, baseline = -.1cm]
\draw (0.5,-1) node[below]{1}--(0.5,1) node[above]{n-c};
\draw (-0.5,0) --(-0.5,1) node[above]{c+1};
\draw(0.5,0.5)--(-0.5,0.2);
\mydotw{(-0.5,0)};
\end{tikzpicture}
\end{equation}.

\normalsize

To verify this, let $z\in \Pi_{1}, w\in \Pi_{c+1}$ and $u\in \Pi_{n-c}$, and let $L^{w,u}_{z}$ and $R^{w,u}_{z}$ denote the coefficient of the basis element $w\otimes u$ of the image of the vector $z$ under the left-hand and right-hand sides of the equations (as in proof of \ref{squarelemma1}). Then 

\begin{align*}
L^{w,u}_{z}&=\begin{cases}
1 & (z,\sigma(u),\sigma(w))\in \mathcal{T} \\
0 & (z,\sigma(u),\sigma(w))\notin \mathcal{T}
\end{cases}
\end{align*}

\begin{align*}
R^{w,u}_{z}&=\begin{cases}
1 & (\sigma(w),z,\sigma(u))\in \mathcal{T} \\
0 & (\sigma(w),z,\sigma(u))\notin \mathcal{T}
\end{cases}
\end{align*}

Thus $L^{w,u}_{z}=R^{w,u}_{z}$ by cyclic invariance (\ref{cyclicinvariance}, Definition \ref{tridefn}).

\end{proof}

\begin{rem}\label{standard} It is straightforward to verify the above construction applied to the degenerate $(q=1)$ triangle presentations of Example \ref{degenerate} works in characteristic $0$. We call these fiber functors \textit{standard}. Note that even when $n$ is odd so that $\text{Web}(\text{SL}^{-}_{n})=\text{Web}(\text{SL}^{+}_{n})$, the fiber functors defined by these degenerate triangle presentations are not the usual ones arising from the forgetful functor $\text{Web}(\text{SL}^{+}_{n})\rightarrow \text{Vec}$ as described above (Remark \ref{tilting}). They have a purely combinatorial flavor. They do preserve the classical dimensions of objects.
\end{rem}

\subsection{Image of the crossing}\label{crossing}

 Even though the crossing generator in the category $\text{Web}(\text{SL}^{-}_{n})$ is not a braiding in general, it still satisfies the Yang-Baxter equation (since it is a braiding in the non-full subcategory $\text{PolyWeb}(\text{GL}_{n})$ ). Our functors do not map this crossing to the usual factor swapping permutation (or any signed twistings of this) and are significantly more complicated in general. Therefore the image of the crossing of any object with itself yields interesting solutions to the (quantum) Yang-Baxter equation which are in addition \textit{involutive} (their square is the identity), which to our knowledge are typically studied in the context of set-theoretical solutions to the Yang-Baxter equation \cite{MR1722951}. We will provide a description of the image of the braid as a linear map. First recall

 \medskip

 \begin{equation}
\check{R}:=\begin{tikzpicture}[scale=0.7, baseline = -.1cm]
\draw  (-0.5,-1) node[below] {1} -- (0.5,1);
\draw  ( 0.5,-1)node[below] {1} -- (-0.5,1) ;
\end{tikzpicture}=\ \ \  \begin{tikzpicture}[scale=0.7, baseline = -.1cm]
\draw (-0.5,-1) node[below]{1} -- (0,-0.5)  ;
\draw (0.5,-1) node[below]{1} -- (0,-0.5) ;
\draw (0,-0.5) -- (0,0.5) ;
\draw (0,0.5) -- (-0.5,1) ;
\draw (0,0.5) -- (0.5,1) ;
\end{tikzpicture}\ \ \ \ 
- \begin{tikzpicture}[scale=0.7, baseline = -.1cm]
\draw (-0.5,-1) node[below] {1} -- (-0.5,1) ;
\draw (0.5,-1) node[below] {1} -- (0.5,1);
\end{tikzpicture}
\end{equation}

\medskip

Let $\mathcal{T}$ be an $n=3$ triangle presentation over a projective plane $\Pi=\Pi_{1}\cup \Pi_{2}$. Let $(u,v),(z,w)\in \Pi_{1}\times \Pi_{1}$. We write $(u,v)\approx (z,w)$ if there exists $s\in \Pi_{1}$ with $(u,v,s),(z,w,s)\in \mathcal{T}$. Note such an $s$ is unique if it exists. 

To get a feeling for this relation consider the setup from Example \ref{numbertheoryex}, and let $D\subseteq \mathbbm{Z}/N\mathbbm{Z}$ denote a standard planar difference set. Then for $m,n\in \Pi_{1}=\mathbbm{Z}/N\mathbbm{Z}$, $\sigma(m)\sim n$ if and only if $n-m\in D.$ Then $(m,n)\approx (m^{\prime}, n^{\prime})$ if and only if the following equation linear is satisfied:

$$m+(q+1)(m-n)=m^{\prime}+(q+1)(m^{\prime}-n^{\prime})$$

We can explicitly write $\check{R}$ in terms of $\approx$ as follows:

\begin{align*}\displaystyle
\check{R}(u\otimes v) &=\begin{cases}
\displaystyle \sum_{\substack{(z,w)\approx (u,v)\\
                  (z,w)\ne (uv)}} z\otimes w \ \ \ \  & \sigma(u)\sim v \\
                  & \\
-u\otimes v & \sigma(u)\nsim v \\
\end{cases}
\end{align*}

Notice that the summation has exactly $q$ non-zero terms when $\sigma(u)\sim v$ and only one non-zero term otherwise. If we picked an ordering on the basis vectors $\Pi_{1}\times \Pi_{1}$, then the associated matrix would be sparse: each column either has $-1$ on the diagonal or exactly $q$ off diagonal terms. There are $(q^{2}+q+1)^{2}$ columns, so the density (or ratio of non-zero matrix entries to all entries) is less then $\frac{q}{(q^{2}+q+1)^{2}}$. It would be interesting to engage in a systematic study of these solutions.

In the introduction, we claimed our solutions where ``positive characteristic'' deformations of easy solutions. We interpret this to mean that when applied to the degenerate triangle presentation of Example \ref{degenerate}, we should obtain a well known, boring solution to the Yang-Baxter equation. To describe these solutions, let $V$ be a vector space, and let $B:=\{v_1, \cdots, v_{n}\}$. Let $\epsilon_{1},\epsilon_{2}\in \{\pm \}$. Then we can define $\check{R}_{\epsilon_1,\epsilon_{2}}(v_{1}\otimes v_{2})$

\begin{align*}\displaystyle
\check{R}_{\epsilon_1,\epsilon_{2}}(v_{1}\otimes v_{2}) &=\begin{cases}
 \epsilon_{1} v_{2}\otimes v_{1} \ \ \ \  & v_{1}\ne v_{2} \\
                     & \\
\epsilon_{2} v_{1}\otimes v_{2}          & v_{1}=v_{2} \\
\end{cases}
\end{align*}

$R_{++}$ and $R_{--}$ are the standard swap solution $P$ and its negative. Clearly $\check{R}^{2}=\text{id}_{V\otimes V}$. To see that this is in general solution, it is well known that $\check{R}$ satisfies our version of the Yang-Baxter equation if and only if $R:=P\circ \check{R}$ satisfies the ``scattering matrix'' version of the equation:

$$ R_{12} R_{13} R_{23}=R_{23} R_{13} R_{12} $$

\noindent where the pair of subscripts denote which factors the operator $R\in \text{End}(V\otimes V)$ acts on in $V\otimes V\otimes V$. But in our case, each $R_{ij}$ is diagonal with respect to the basis $B\times B\times B$, so the equation is satisfied. Applying our construction to the degenerate triangle presentation for $n=3$ and $a=1$ yields $R_{+-}$. In our positive characteristic examples, verifying the Yang-Baxter directly relation is much less trivial.

\bibliographystyle{amsalpha}
{\footnotesize{
\bibliography{bibliography}
}}

\end{document}